\documentclass{article}
\usepackage{amsmath, amsthm, amssymb, amsfonts}
\usepackage{bm}
\usepackage{empheq}
\usepackage[T1]{fontenc}
\usepackage{xcolor}

\numberwithin{equation}{section}

\usepackage{geometry}
\usepackage[shortlabels]{enumitem}
\usepackage[normalem]{ulem}

\usepackage{algorithm}
\usepackage{algorithmic}
\usepackage{adjustbox}
\usepackage{booktabs}
\usepackage{multirow}
\usepackage{siunitx}
\usepackage{array}
\usepackage{graphicx}
\usepackage{caption}
\usepackage{subcaption}
\usepackage{tikz}
\usetikzlibrary{
  positioning,
  fit,
  arrows.meta,
  calc,
  decorations.markings,
  backgrounds,
  shadows
}

\graphicspath{ {./img/} }

\usepackage{cite}

\usepackage{hyperref}
\hypersetup{
    colorlinks=true,
    linkcolor=blue
}

\usepackage{thmtools}
\usepackage{thm-restate}

\usepackage[capitalize, noabbrev, nameinlink]{cleveref}

\declaretheorem[name=Theorem, numberwithin=section]{theorem}
\declaretheorem[name=Proposition, sibling=theorem]{proposition}
\declaretheorem[name=Corollary,   sibling=theorem]{corollary}
\declaretheorem[name=Lemma,       sibling=theorem]{lemma}

\declaretheorem[name=Definition,  style=definition, sibling=theorem]{definition}

\declaretheorem[name=Remark,      style=remark,     sibling=theorem]{remark}

\crefdefaultlabelformat{#2\textup{#1}#3}

\crefformat{equation}{\textup{#2(#1)#3}}
\crefrangeformat{equation}{\textup{#3(#1)#4--#5(#2)#6}}
\crefmultiformat{equation}{\textup{#2(#1)#3}}{ and \textup{#2(#1)#3}}{, \textup{#2(#1)#3}}{, and \textup{#2(#1)#3}}

\Crefformat{equation}{#2Equation~\textup{(#1)}#3}
\Crefrangeformat{equation}{Equations~\textup{#3(#1)#4--#5(#2)#6}}
\Crefmultiformat{equation}{Equations~\textup{#2(#1)#3}}{ and \textup{#2(#1)#3}}{, \textup{#2(#1)#3}}{, and \textup{#2(#1)#3}}

\crefname{section}{section}{sections}
\Crefname{section}{Section}{Sections}

\begin{document}

\begin{center}

{\bf \large
    \textsc{From P\'olya's Conditions to a Complete Characterization of the $L^2$ Convergence of Hyperinterpolation}
}

\vspace{0.3cm}
\scshape
\renewcommand{\thefootnote}{\fnsymbol{footnote}}

Congpei An\footnote{School of Mathematics and Statistics, Guizhou University, Guiyang 550025, China. This author was supported by NSFC (No. 12371099). Email: andbachcp@gmail.com}
\qquad
Xiannan Hu\footnote{Department of Mathematics, The University of Hong Kong, Hong Kong, China. Email: hans0711@connect.hku.hk}
\qquad
Xiaoming Yuan\footnote{Department of Mathematics, The University of Hong Kong, Hong Kong, China. This author was supported by the Croucher Senior Fellowship and the GRF 17305825. Email: xmyuan@hku.hk}

\renewcommand{\thefootnote}{\arabic{footnote}}
\setcounter{footnote}{0}

\vspace{0.3cm}
% August 1, 2026 % Uploaded to arXiv
October 8, 2026 % First revision
\end{center}

\begin{abstract}
It has remained open to identify the necessary and sufficient conditions for the $L^2$ convergence of hyperinterpolation since it was introduced by Sloan in 1995. We show that the $L^1$--$L^2$ Marcinkiewicz--Zygmund (MZ) condition, together with the asymptotic functional approximation property for polynomials, is the answer. We further prove that the optimal $L^1$--$L^2$ MZ constant coincides with the operator norm of the hyperinterpolation operator, and it admits a natural Banach space duality interpretation. With an explicit construction, we also show that P\'{o}lya's classical conditions for quadrature convergence are not sufficient for the $L^2$ convergence of hyperinterpolation. This reveals a fundamental distinction between the convergence of linear functionals (quadrature formulas) and that of linear operators (hyperinterpolation operators). We establish a strict logical hierarchy for the stability and accuracy conditions governing the convergence of quadrature and hyperinterpolation.
\end{abstract}

\vspace{0.3cm}
\noindent\textbf{\textsc{Keywords}}: hyperinterpolation, P\'olya's conditions, Marcinkiewicz--Zygmund inequalities, \\
\phantom{\textbf{\textsc{Keywords}}\;\,} functional approximation, $L^2$ convergence, sphere.

\vspace{0.3cm}

\noindent\textbf{\textsc{AMS Subject Classifications}}: 41A35, 41A17, 41A55, 65D15, 65D32.

\section{Introduction}

Hyperinterpolation was introduced by Sloan in his seminal paper \cite{sloan1995polynomial} and subsequently developed in, e.g., \cite{dai2006generalized,hesse2006hyperinterpolation,reimer2000hyperinterpolation,sloan2000constructive,reimer2012multivariate,sloan2012filtered,reimer2002generalized}. It is a constructive polynomial approximation method for continuous functions on a bounded region $\Omega$ of finite measure, which is either the closure of a connected open domain or a smooth closed manifold. It is obtained from the orthogonal projection onto $\mathbb{P}_n(\Omega)$, the space of polynomials on $\Omega$ of degree at most $n$, by replacing the integrals in the $L^2$ inner products with quadrature evaluations. Within this general framework, the unit sphere is a canonical and particularly important setting, owing to its well-understood spherical harmonic structure and its central role in the approximation and analysis of global and scattered data. We therefore focus throughout this paper on the unit sphere $\mathbb{S}^2 := \{\mathbf{x} \in \mathbb{R}^3 \,|\, \|\mathbf{x}\|_2 = 1\}.$

In the literature, e.g., \cite{sloan1995polynomial,sloan2000constructive}, the quadrature formula $Q$ is classically assumed to integrate all polynomials of degree up to $2n$ exactly with positive weights. Under these assumptions, the hyperinterpolation operator $L_n: C(\mathbb{S}^2) \to \mathbb{P}_n(\mathbb{S}^2)$ converges in the following $L^2$ sense \cite[Theorem~1]{sloan1995polynomial}:
\begin{equation} \label{eq:hyperinterpolation-convergence}
    \lim_{n\to\infty} \|L_n f - f\|_{L^2(\mathbb{S}^2)} = 0 \quad \forall\, f\in C(\mathbb{S}^2).
\end{equation}
In \cite{an2024bypassing}, it was shown that algebraic exactness is not essential for the $L^2$ convergence of hyperinterpolation, while positivity of the quadrature weights remains. For real-world applications with scattered data such as MAGSAT magnetic field measurements and microwave background observations \cite{ whaler1994downward,mandea2006magnetic,olsen2006chaos,bennett1996four,jarosik2011seven}, quadrature nodes are often prescribed by physical and instrumental constraints. For such data, the mentioned theoretical assumptions may not be satisfied and it becomes infeasible to construct quadrature formulas with positive weights and high algebraic exactness simultaneously. Instead, despite that negative weights are commonly regarded as a source of instability for numerical integration, quadrature formulas with lower exactness and signed weights should be considered for practical computation.

We are thus motivated to study hyperinterpolation with neither the positivity of weights nor the exactness assumptions of quadrature formulas. For this more general setting, recently in \cite[(5.8)]{an2026optimizationapproachweightcollocation} we presented a stability-accuracy decomposition for the hyperinterpolation error $\|L_nf - f\|_{L^2(\mathbb{S}^2)}$, where the stability term is governed by the $L^2$ Marcinkiewicz–Zygmund (MZ) condition and the accuracy term is determined by the $L^2$ MZ inequality. Still, a fundamental problem remains open: 
\begin{quotation}
    \textit{
    What exactly are the necessary and sufficient conditions for the $L^2$ convergence of hyperinterpolation \eqref{eq:hyperinterpolation-convergence}?}
\end{quotation}

Recall that P\'olya's classical theorem on quadrature convergence in \cite{polya1933konvergenz} states that a sequence of quadrature formulas converges for every continuous function if and only if the sequence of quadrature formulas satisfy the uniform P\'olya stability condition (cf. \eqref{eq:polya-cond}) and the asymptotic approximation property for polynomials (cf. \eqref{eq:approximation-property}). The former is a stability condition equivalent to the uniform boundedness of the operator norms of the quadrature functionals (cf. \cref{rem:quadrature-stability}), whereas the latter is an accuracy condition requiring asymptotic exactness on all polynomials. Since hyperinterpolation is obtained by replacing the $L^2$ inner products in the orthogonal projection with quadrature evaluations, it seems intuitive to conjecture that P\'olya's conditions can also guarantee the $L^2$ convergence of hyperinterpolation. 

We show that this conjecture is false, by explicitly constructing a sequence of quadrature formulas $\{Q_n\}_{n\in\mathbb{N}}$ satisfying P\'olya's conditions for which their associated hyperinterpolation operators $L_n$ have unbounded operator norms (cf. \cref{thm:hyper-2-neg}). Then, by the uniform boundedness principle, there exists a function $f \in C(\mathbb{S}^2)$ such that $\limsup_{n\to\infty}\|L_n f\|_{L^2(\mathbb{S}^2)} = \infty$, where hyperinterpolation fails to converge. This counterexample shows that the convergence of linear functionals (represented here by quadrature formulas) is strictly weaker than the convergence of the corresponding linear operators (represented here by hyperinterpolation operators). Indeed, both the stability and accuracy conditions of P\'olya considered in \cite{polya1933konvergenz} are insufficient. We first construct a sequence of quadrature formulas that satisfies the uniform P\'olya stability condition but does not control the stability of higher degree polynomials (cf. \cref{prop:Polya-does-not-imply-L1-L2-MZ}). Then, we construct another sequence that satisfies the asymptotic approximation property for polynomials but the hyperinterpolation does not converge even for the constant function (cf. \cref{prop:AP-does-not-imply-FAP}).

Inspired by the classical $L^1$ and $L^2$ MZ conditions in \cite{gia2009localized} for the stability of hyperinterpolation, we propose the $L^1$--$L^2$ MZ condition (cf. \eqref{eq:L1-L2-MZ}) to connect the discrete $L^1$ seminorm with the continuous $L^2$ norm. We further show that the $L^1$--$L^2$ MZ condition is also necessary for the stability of hyperinterpolation (cf. \cref{thm:L1-L2-MZ-equivalent-to-hyperinterpolation-stability}). The optimal $L^1$--$L^2$ MZ constant of a quadrature formula coincides with the operator norm of its associated hyperinterpolation operator (cf. \cref{cor:sharp-MZ-constant}). This norm identity also admits a natural Banach space duality interpretation in the following sense: the discrete $L^1$ expression is the total variation norm of an atomic signed measure, whereas the $L^2$ norm reflects the Hilbert-space geometry of the polynomial space (cf. \Cref{sec:duality}). The $L^1$--$L^2$ MZ condition, combined with the asymptotic functional approximation property for polynomials \eqref{eq:functional-approximation-property}, yields a complete characterization of the $L^2$ convergence of hyperinterpolation (cf. \cref{thm:convergence}). 

These results lead to a strict logical hierarchy among the stability and accuracy conditions considered in this paper (cf. \cref{thm:hierarchy}). In particular, the classical $L^1$ and $L^2$ MZ conditions are strictly stronger than the $L^1$--$L^2$ MZ condition, which is in turn strictly stronger than the uniform P\'olya stability condition. Also, the asymptotic functional approximation property \eqref{eq:functional-approximation-property} is strictly stronger than the asymptotic approximation property \eqref{eq:approximation-property}. The schematic diagram in \cref{fig:hierarchy} summarizes the strict logical hierarchy of stability and accuracy conditions pertinent to the convergence of quadrature and hyperinterpolation.

\begin{figure}
\centering

\definecolor{TideBlue}{RGB}{0,60,101}
\definecolor{TideGrey}{RGB}{235,235,230}
\definecolor{TideRed}{RGB}{190,60,55}
\definecolor{TideYellow}{RGB}{200,155,20}

\tikzset{
  cdimpl/.style={
    -{Stealth[length=2.4mm]},
    line width=0.4pt,
    double,
    double distance=0.7pt
  },
  cdnotimpl/.style={
    -{Stealth[length=2.4mm]},
    line width=0.4pt,
    double,
    double distance=0.7pt,
    decoration={
      markings,
      mark=at position 0.5 with {
        \draw[-,line width=0.5pt]
          (-0.9mm,-1.4mm)--(0.9mm,1.4mm);
      }
    },
    postaction=decorate
  },
  mzimpl/.style={
    -{Stealth[length=2.1mm]},
    line width=0.3pt,
    double,
    double distance=0.55pt
  },
  mznotimpl/.style={
    -{Stealth[length=2.1mm]},
    line width=0.3pt,
    double,
    double distance=0.55pt,
    decoration={
      markings,
      mark=at position 0.5 with {
        \draw[-,line width=0.4pt]
          (-0.8mm,-1.2mm)--(0.8mm,1.2mm);
      }
    },
    postaction=decorate
  },
  cdequivbase/.style={
    {Stealth[length=2.4mm]}-{Stealth[length=2.4mm]},
    line width=0.6pt,
    double,
    double distance=0.8pt,
    shorten <=-1pt,
    shorten >=-1pt
  },
  cdequivhyp/.style={
    cdequivbase
  },
  cdequivquad/.style={
    cdequivbase
  },
  condbase/.style={
    draw=black!50,
    rounded corners=3pt,
    line width=0.5pt,
    inner sep=5pt,
    drop shadow={
      shadow xshift=1pt,
      shadow yshift=-1pt,
      opacity=0.45,
      fill=black!60
    }
  },
  condhyp/.style={
    condbase,
    top color=white,
    bottom color=TideBlue!28
  },
  condquad/.style={
    condbase,
    top color=white,
    bottom color=TideRed!30
  },
  condmz/.style={
    condbase,
    top color=white,
    bottom color=TideYellow!22
  },
  grouphyp/.style={
    draw=TideBlue!70!black,
    dashed,
    rounded corners=7pt,
    inner sep=9pt,
    line width=0.8pt,
    fill=TideBlue!8
  },
  groupquad/.style={
    draw=TideRed!70!black,
    dashed,
    rounded corners=7pt,
    inner sep=9pt,
    line width=0.8pt,
    fill=TideRed!8
  },
  sepline/.style={
    dashed,
    draw=TideGrey!60!black,
    line width=0.5pt,
    dash pattern=on 3pt off 2.5pt
  },
  bandlab/.style={
    rotate=90,
    font=\footnotesize\sffamily,
    text=black
  },
  reflab/.style={
    font=\scriptsize
  },
  pluslab/.style={
    font=\large\bfseries,
    fill=white,
    fill opacity=0.9,
    text opacity=1,
    inner sep=1.5pt,
    rounded corners=1pt
  }
}

\newlength{\mzboxwd}
\setlength{\mzboxwd}{6cm}

\newlength{\opboxwd}
\setlength{\opboxwd}{4cm}

\adjustbox{max width=\linewidth}{%
\begin{tikzpicture}[
  x=1cm,
  y=1cm,
  every node/.style={align=center}
]

  % ============================================================
  % CONVERGENCE
  % ============================================================

  \node[
    condhyp,
    minimum width=\mzboxwd
  ] (hiconv) at (3.5,7.0) {
    \itshape Hyperinterpolation Convergence\\[2pt]
    \normalfont $\|L_n f-f\|_{L^2(\mathbb{S}^2)}\to0$
  };

  \node[
    condquad,
    minimum width=\mzboxwd
  ] (qconv) at (13.0,7.0) {
    \itshape Quadrature Convergence\\[2pt]
    \normalfont $|Q_n(f)-I(f)|\to0$
  };

  % ============================================================
  % ACCURACY
  % ============================================================

  \node[
    condhyp,
    text width=\mzboxwd
  ] (fap) at (3.5,3.55) {
    $\|L_n p-p\|_{L^2(\mathbb{S}^2)}\to0$
    \eqref{eq:functional-approximation-property}
  };

  \node[
    condquad,
    text width=\mzboxwd
  ] (approx) at (13.0,3.55) {
    $|Q_n(p)-I(p)|\to0$
    \eqref{eq:approximation-property}
  };

  % ============================================================
  % STABILITY: CONDITIONS
  % ============================================================

  \node[
    condhyp,
    text width=\mzboxwd
  ] (l1l2) at (3.5,1.90) {
    $L^1$--$L^2$ MZ \eqref{eq:L1-L2-MZ}
  };

  \node[
    condquad,
    text width=\mzboxwd
  ] (polya) at (13.0,1.90) {
    Uniform P\'olya Stability \eqref{eq:polya-cond}
  };

  % ============================================================
  % STABILITY: OPERATOR-NORM CHARACTERIZATIONS
  % ============================================================

  \node[
    condhyp,
    text width=\opboxwd
  ] (opLn) at (3.5,-1.5) {
    $\displaystyle
      \sup^{\phantom{n\in\mathbb{N}}}_{n\in\mathbb{N}}
      \|L_n\|_{C(\mathbb{S}^2) \to L^2(\mathbb{S}^2)}\leq C$
  };

  \node[
    condquad,
    text width=\opboxwd
  ] (opQn) at (13.0,-1.5) {
    $\displaystyle
      \sup^{\phantom{n\in\mathbb{N}}}_{n\in\mathbb{N}}
      \|Q_n\|_{C(\mathbb{S}^2) \to\mathbb{R}}\leq M$
  };

  % ============================================================
  % STABILITY: CLASSICAL MZ CONDITIONS
  % ============================================================

  \node[condmz] (l1mz) at (1.00,-3) {
    $L^1$ MZ \eqref{eq:L1-MZ}
  };

  \node[condmz] (l2mz) at (6.00,-3) {
    $L^2$ MZ \eqref{eq:L2-MZ}
  };

  \coordinate (l1entry) at
    (l1mz.north |- l1l2.south);

  \coordinate (l2entry) at
    (l2mz.north |- l1l2.south);

  % Combination symbols.
  \node[pluslab,text=TideBlue!65!black]
    at (3.5,2.70) {\textbf{+}};

  \node[pluslab,text=TideRed!70!black]
    at (13.0,2.70) {\textbf{+}};

  % ============================================================
  % ROW BOUNDARIES
  % ============================================================

  \coordinate (sepU) at (0,5.15);
  \coordinate (sepM) at (0,2.80);

  \coordinate (topGuide)    at (-1.0,7.45);
  \coordinate (sepUGuide)   at (-1.0,5.15);
  \coordinate (sepMGuide)   at (-1.0,2.80);
  \coordinate (bottomGuide) at (-1.0,-2.90);

  \coordinate (convLabelPos) at
    ($(topGuide)!0.5!(sepUGuide)$);

  \coordinate (accLabelPos) at
    ($(sepUGuide)!0.5!(sepMGuide)$);

  \coordinate (stabLabelPos) at
    ($(sepMGuide)!0.5!(bottomGuide)$);

  \node[bandlab] (lconv) at (convLabelPos) {
    Convergence
  };

  \node[bandlab] (laccu) at (accLabelPos) {
    Accuracy
  };

  \node[bandlab] (lstab) at (stabLabelPos) {
    Stability
  };

  % ============================================================
  % DASHED GROUP BOXES AND SEPARATORS
  % ============================================================

  \begin{scope}[on background layer]

    \node[grouphyp,fit=(fap)(l1l2)]
      (boxhyp) {};

    \node[groupquad,fit=(approx)(polya)]
      (boxquad) {};

    \node[fit=(lconv)(laccu)(lstab),inner sep=0pt]
      (bL) {};

    \draw[sepline]
      ([xshift=-6pt]bL.west |- sepU)
      --
      (boxquad.east |- sepU);

    \draw[sepline]
      ([xshift=-6pt]bL.west |- sepM)
      --
      (boxquad.east |- sepM);

  \end{scope}

  % ============================================================
  % MIDPOINTS
  % ============================================================

  \coordinate (l1mid) at
    ($(l1mz.north)!0.5!(l1entry)$);

  \coordinate (l2mid) at
    ($(l2mz.north)!0.5!(l2entry)$);

  \coordinate (hconvmid) at
    ($(boxhyp.north)!0.5!(hiconv.south)$);

  \coordinate (qconvmid) at
    ($(boxquad.north)!0.5!(qconv.south)$);

  \coordinate (lnormmid) at
    ($(l1l2.south)!0.5!(opLn.north)$);

  \coordinate (qnormmid) at
    ($(polya.south)!0.5!(opQn.north)$);

  % ============================================================
  % ACCURACY RELATIONS
  % ============================================================

  \draw[cdimpl,transform canvas={yshift=4pt}]
    (fap.east)--(approx.west)
    node[midway,above,reflab] {
      \cref{prop:fap-implies-ap}
    };

  \draw[cdnotimpl,transform canvas={yshift=-4pt}]
    (approx.west)--(fap.east)
    node[midway,below,reflab] {
      \cref{prop:AP-does-not-imply-FAP}
    };

  % ============================================================
  % CONVERGENCE RELATIONS
  % ============================================================

  \draw[cdimpl,transform canvas={yshift=4pt}]
    (hiconv.east)--(qconv.west)
    node[midway,above,reflab] {};

  \draw[cdnotimpl,transform canvas={yshift=-4pt}]
    (qconv.west)--(hiconv.east)
    node[midway,below,reflab] {
      \cref{thm:hyper-2-neg}
    };

  % ============================================================
  % L1--L2 MZ AND PÓLYA
  % ============================================================

  \draw[cdimpl,transform canvas={yshift=4pt}]
    (l1l2.east)--(polya.west)
    node[midway,above,reflab] {
      \cref{lem:L1-L2-MZ-implies-Polya}
    };

  \draw[cdnotimpl,transform canvas={yshift=-4pt}]
    (polya.west)--(l1l2.east)
    node[midway,below,reflab] {
      \cref{prop:Polya-does-not-imply-L1-L2-MZ}
    };

  % ============================================================
  % CLASSICAL MZ RELATIONS
  % ============================================================

  \draw[mzimpl,transform canvas={xshift=-5pt}]
    (l1mz.north)--(l1entry);

  \draw[mznotimpl,transform canvas={xshift=5pt}]
    (l1entry)--(l1mz.north);

  \node[reflab,rotate=90]
    at ([xshift=-12pt, yshift=20pt]l1mid) {
      \cref{lem:L1-MZ-implies-L1-L2-MZ}
    };

  \node[reflab,rotate=90]
    at ([xshift=12pt, yshift=20pt]l1mid) {
      \cref{thm:hierarchy}
    };

  \draw[mzimpl,transform canvas={xshift=-5pt}]
    (l2mz.north)--(l2entry);

  \draw[mznotimpl,transform canvas={xshift=5pt}]
    (l2entry)--(l2mz.north);

  \node[reflab,rotate=90]
    at ([xshift=-12pt, yshift=20pt]l2mid) {
      \cref{lem:L2-MZ-implies-L1-L2-MZ}
    };

  \node[reflab,rotate=90]
    at ([xshift=12pt, yshift=20pt]l2mid) {
      \cref{thm:hierarchy}
    };

  % ============================================================
  % CONVERGENCE EQUIVALENCES
  % ============================================================

  \draw[cdequivhyp]
    (boxhyp.north)--(hiconv.south);

  \node[reflab,rotate=90]
    at ([xshift=-10pt]hconvmid) {
      \cref{thm:convergence}
    };

  \draw[cdequivquad]
    (boxquad.north)--(qconv.south);

  \node[reflab,rotate=90]
    at ([xshift=-10pt]qconvmid) {
      \cref{thm:polya}
    };

  % ============================================================
  % OPERATOR-NORM EQUIVALENCES
  % ============================================================

  \draw[cdequivhyp]
    (l1l2.south)--(opLn.north);

  \node[reflab,rotate=90]
    at ([xshift=-12pt]lnormmid) {
      \cref{thm:L1-L2-MZ-equivalent-to-hyperinterpolation-stability}
    };

  \draw[cdequivquad]
    (polya.south)--(opQn.north);

  \node[reflab,rotate=90]
    at ([xshift=-12pt]qnormmid) {
      \cref{rem:quadrature-stability}
    };

\end{tikzpicture}
}

\caption{The complete landscape of various stability and accuracy conditions related to quadrature and $L^2$ hyperinterpolation convergence.}
\label{fig:hierarchy}
\end{figure}
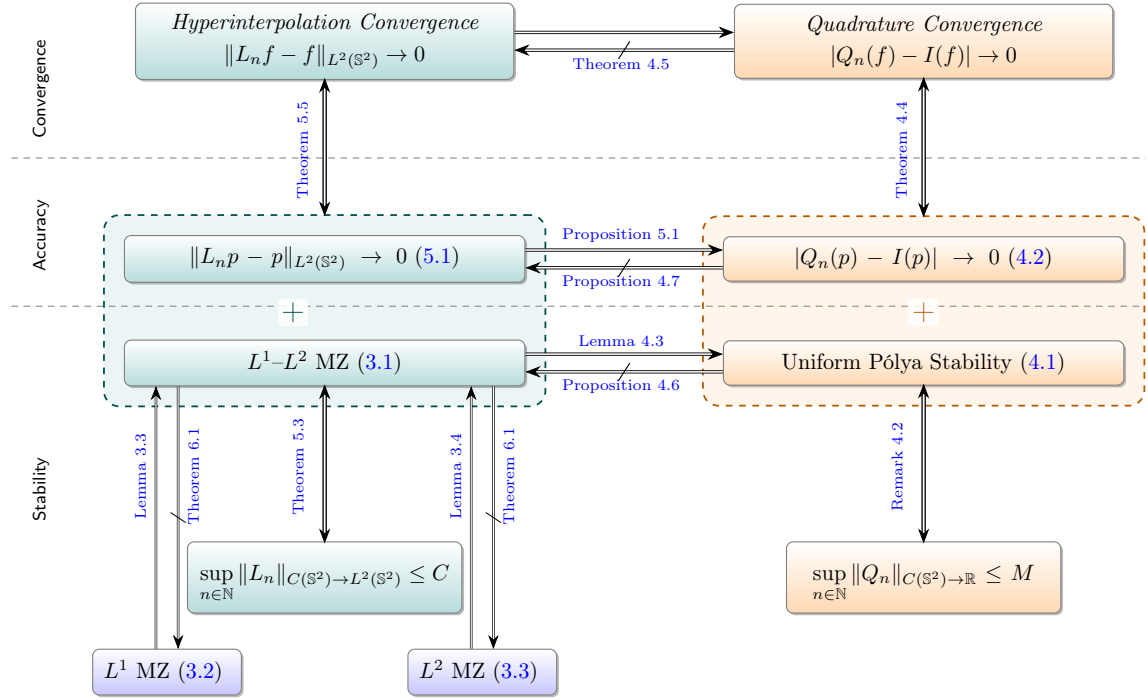

The remainder of this paper is organized as follows. In \Cref{sec:preliminaries}, we recall essential preliminaries. In \Cref{sec:existence}, we propose the $L^1$--$L^2$ MZ condition and prove its existence. Its connection to classical MZ theory is also discussed. \Cref{sec:does-not-imply} presents a counterexample demonstrating that P\'olya's conditions on quadrature convergence are not sufficient to guarantee the $L^2$ convergence of hyperinterpolation, and further discusses the independent limitations of the stability and accuracy conditions of P\'olya. In \Cref{sec:equivalence}, we establish the necessary and sufficient conditions for the $L^2$ convergence of hyperinterpolation. This is the main result of the paper. \Cref{sec:hierarchy} delineates the strict logic hierarchy formed by various stability and accuracy conditions. \Cref{sec:validation} discusses some computational aspects related to our theoretical assertions, and presents numerical validations of the hierarchy formed by the stability conditions. Finally, we make some concluding remarks in \Cref{sec:conclusion}. 

Note that the present paper is concerned exclusively with the $L^2$ convergence. Indeed, the operator norm of any linear projection, including hyperinterpolation, in the setting $C \to C$ grows at least on the order of $\sqrt{n}$ \cite[Theorem~6.6]{reimer2012multivariate}. Questions of uniform convergence require different stability mechanisms, and they are not pursued here.

\section{Preliminaries}
\label{sec:preliminaries}

We equip $\mathbb{S}^2$ with the Lebesgue surface measure $\omega$, normalized so that $|\mathbb{S}^2| = \int_{\mathbb{S}^2}\,d\omega(\mathbf{x}) = 4\pi$. The intrinsic geometry of $\mathbb{S}^2$ is given by the geodesic distance $\mathrm{dist}(\mathbf{x}, \mathbf{y}) = \arccos(\langle \mathbf{x}, \mathbf{y}\rangle)$, where $\langle \cdot, \cdot\rangle$ denotes the standard Euclidean inner product in $\mathbb{R}^3$.

\subsection{Spherical Harmonics}
Let $C(\mathbb{S}^2)$ denote the space of real-valued continuous functions on $\mathbb{S}^2$ equipped with the uniform norm $\|f\|_{\infty} := \sup_{\mathbf{x} \in \mathbb{S}^2} |f(\mathbf{x})|$. For $p \in [1, \infty)$, we denote the standard Lebesgue space by $L^p(\mathbb{S}^2)$ equipped with the norm $\|f\|_{L^p(\mathbb{S}^2)} := (\int_{\mathbb{S}^2} |f(\mathbf{x})|^p\,d\omega(\mathbf{x}))^{1/p}$. In particular, $L^2(\mathbb{S}^2)$ is a Hilbert space with the inner product
$\langle f, g \rangle_{L^2(\mathbb{S}^2)} := \int_{\mathbb{S}^2} f(\mathbf{x})g(\mathbf{x})\,d\omega(\mathbf{x})$.

\textit{Spherical harmonics} \cite{muller2006spherical} are the restrictions of harmonic homogeneous polynomials in $\mathbb{R}^3$ to $\mathbb{S}^2$. Let $\mathbb{H}_\ell$ denote the space of spherical harmonics of degree $\ell \in \mathbb{N}_0$. It is known that $\mathrm{dim}(\mathbb{H}_\ell) = 2\ell + 1$. We choose an orthonormal basis for each $\mathbb{H}_\ell$
\begin{equation*}
    \{Y_{\ell, k} \,|\, k = -\ell, \ldots, -1, 0, 1, \ldots, \ell\}.
\end{equation*}
The spherical harmonics satisfy the \textit{addition theorem}:
\begin{equation} \label{eq:addition-theorem}
    \sum_{k = -\ell}^{\ell} Y_{\ell, k}(\mathbf{x})Y_{\ell, k}(\mathbf{y}) = \frac{2\ell + 1}{4\pi}P_\ell(\langle \mathbf{x}, \mathbf{y}\rangle) \quad \forall\, \mathbf{x}, \mathbf{y} \in \mathbb{S}^2,
\end{equation}
where $P_{\ell}$ is the Legendre polynomial of degree $\ell$ normalized such that $P_\ell(1) = 1$.

Let $\mathbb{P}_n(\mathbb{S}^2) = \bigoplus_{\ell=0}^n \mathbb{H}_\ell$ be the space of spherical polynomials of degree at most $n$, and let $\mathbb{P}(\mathbb{S}^2) = \bigoplus_{\ell=0}^\infty \mathbb{H}_\ell$ be the space of all spherical polynomials. Because spherical harmonics of different degrees are mutually orthogonal, the collection $\{Y_{\ell, k}\,|\, 0 \leq \ell \leq n, |k| \leq \ell\}$ forms an orthonormal basis for $\mathbb{P}_n(\mathbb{S}^2)$, implying that $\mathrm{dim}(\mathbb{P}_n(\mathbb{S}^2)) = (n+1)^2$. Furthermore, the set of all spherical harmonics forms a complete orthonormal system for $L^2(\mathbb{S}^2)$.

Throughout this paper, we denote by $e \in \mathbb{P}_0(\mathbb{S}^2)$ the constant polynomial $e(\mathbf{x}) \equiv 1$. Note that $e(\mathbf{x}) = \sqrt{4\pi}Y_{0,0}(\mathbf{x})$.

\subsection{Quadrature Formulas}

Let
\begin{equation*}
    I: C(\mathbb{S}^2)\to\mathbb{R},
    \qquad
    I(f) := \int_{\mathbb{S}^2}f(\mathbf{x})\,d\omega(\mathbf{x})
\end{equation*}
denote the integration functional. Let $X=\{\mathbf{x}_1,\ldots,\mathbf{x}_N\}\subseteq\mathbb{S}^2$ be a set of $N$ distinct nodes and $\mathbf{w} = (w_1,\ldots,w_N)^\top\in\mathbb{R}^N$ be a vector of real weights. The \textit{quadrature formula} associated with the nodes $X$ and weights $\mathbf{w}$ is the linear functional
\begin{equation*}
    Q_{X,\mathbf{w}}: C(\mathbb{S}^2)\to\mathbb{R},
    \qquad
    Q_{X,\mathbf{w}}(f):=\sum_{j=1}^{N}w_jf(\mathbf{x}_j).
\end{equation*}
When the nodes and weights are clear from the context, we write simply $Q$. Its operator norm is
\begin{equation}\label{eq:quadrature-functional-norm}
    \|Q\|_{C(\mathbb{S}^2)\to\mathbb{R}} := \sup_{\|f\|_\infty\leq1}|Q(f)| = \sum_{j=1}^{N}|w_j|
      = \|\mathbf{w}\|_1.
\end{equation}
Equivalently, $\|Q\|_{C(\mathbb{S}^2)\to\mathbb{R}}$ is the dual norm of $Q$ in $C(\mathbb{S}^2)^*$. A quadrature formula $Q$ is said to be \textit{exact} of degree $m$ if $Q(p)=I(p)$ for all $p\in\mathbb{P}_m(\mathbb{S}^2)$. For each $n \in \mathbb{N}$, let $X^{(n)}=\{\mathbf{x}_1^{(n)},\ldots,\mathbf{x}_{N_n}^{(n)}\}$ be a set of $N_n$ distinct nodes, and let $\mathbf{w}^{(n)}=(w_1^{(n)},\ldots,w_{N_n}^{(n)})^\top \in\mathbb{R}^{N_n}$ be the associated weight vector. Define $Q_n:=Q_{X^{(n)},\mathbf{w}^{(n)}}$. Here, the subscript $n$ serves only as a sequence index and does not, in general, indicate the degree of exactness of $Q_n$.

\subsection{Spherical Designs}
Many constructions of quadrature formulas in this paper are based on spherical designs. A \textit{spherical $n$-design} is a set of points $X^{(n)}$ with the characterizing property that an equal-weight quadrature formula associated with these points integrates all polynomials of degree at most $n$ exactly, that is,
\begin{equation*}
I(p) = \frac{4\pi}{N_n} \sum_{j=1}^{N_n} p(\mathbf{x}_j^{(n)}) \quad \forall\, p \in \mathbb{P}_n(\mathbb{S}^2).
\end{equation*}
We refer to \cite{bondarenko2013optimal} for the existence of spherical $n$-designs. Spherical designs were used in \cite[Section~4.1]{sloan1995polynomial} as the canonical quadrature formulas for hyperinterpolation. In this paper, we use them to construct counterexamples.

\subsection{Hyperinterpolation}

For a quadrature formula $Q$, define a discrete bilinear form $\langle \cdot, \cdot \rangle_{Q}$ on $C(\mathbb{S}^2)$ as 
\begin{equation*}
    \langle f, g \rangle_Q := Q(fg) = \sum_{j=1}^N w_j f(\mathbf{x}_j)g(\mathbf{x}_j), \quad f, g \in C(\mathbb{S}^2).
\end{equation*}
The weights $w_j$ are allowed to be signed in the general setting of this paper. Let $G_n$ be the reproducing kernel of $\mathbb{P}_n(\mathbb{S}^2)$ defined by
\begin{equation*}
    G_n(\mathbf{x}, \mathbf{y}) := 
    \sum_{\ell=0}^n \sum_{k = -\ell}^{\ell} Y_{\ell, k}(\mathbf{x})Y_{\ell, k}(\mathbf{y}) = \sum_{\ell=0}^n \frac{2\ell+1}{4\pi} P_\ell(\langle \mathbf{x}, \mathbf{y}\rangle),
\end{equation*}
where the last equality uses the addition theorem \eqref{eq:addition-theorem}. The following well-known properties of $G_n$ will be utilized throughout our analysis.

\begin{lemma} \label{lem:Gn-properties}
    For all $\mathbf{x},\mathbf{y}\in\mathbb{S}^2$, we have
    \begin{enumerate}[(1), noitemsep]
        \item $\|G_n\|_\infty = G_n(\mathbf{x},\mathbf{x}) = (n+1)^2/(4\pi)$; and
        \item the following reproducing property holds for $G_n$:
        \begin{equation} \label{eq:reproducing-property}
            p(\mathbf{x}) = \int_{\mathbb{S}^2} G_n(\mathbf{x},\mathbf{y})p(\mathbf{y})\,d\omega(\mathbf{y}) \quad \forall\, p \in \mathbb{P}_n(\mathbb{S}^2).
        \end{equation}
    \end{enumerate}
\end{lemma}
\begin{proof}
    To establish the first property, use $|P_\ell(t)| \leq P_\ell(1) = 1$. To establish the second property, use orthonormality of spherical harmonics.
\end{proof}

Following~\cite{sloan1995polynomial}, for a quadrature formula $Q$ and an integer
$n\in\mathbb{N}_0$, we define the \textit{degree-$n$ hyperinterpolation operator} 
\begin{equation*}
    L_{n,Q}:C(\mathbb{S}^2)\to\mathbb{P}_n(\mathbb{S}^2)
\end{equation*}
by replacing the $L^2$ inner products in the orthogonal projection onto
$\mathbb{P}_n(\mathbb{S}^2)$ with the following discrete bilinear form
$\langle\cdot,\cdot\rangle_Q$:
\begin{equation}\label{eq:hyperinterpolation}
    L_{n,Q}f(\mathbf{x}) :=\sum_{\ell=0}^{n}\sum_{k=-\ell}^{\ell} \langle f,Y_{\ell,k}\rangle_QY_{\ell,k}(\mathbf{x}) = \langle f,G_n(\mathbf{x},\cdot)\rangle_Q,
    \quad \mathbf{x}\in\mathbb{S}^2,
\end{equation}
where the last equality follows from the addition theorem \eqref{eq:addition-theorem}. When the quadrature formula is clear from the context, we write simply $L_n$. We equip $\mathbb{P}_n(\mathbb{S}^2)$ with the norm inherited from $L^2(\mathbb{S}^2)$. The operator norm of $L_{n}$ is
\begin{equation}\label{eq:hyperinterpolation-operator-norm}
    \|L_{n}\|_{C(\mathbb{S}^2)\to L^2(\mathbb{S}^2)} :=\sup_{\substack{f\in C(\mathbb{S}^2)\\f\neq0}} \frac{\|L_{n}f\|_{L^2(\mathbb{S}^2)}}{\|f\|_\infty} =\sup_{\|f\|_\infty\leq1}\|L_{n}f\|_{L^2(\mathbb{S}^2)}.
\end{equation}

Hyperinterpolation classically requires the underlying quadrature formula $Q$ to have positive weights and be exact of degree $2n$. Neither assumption is imposed in the general setting considered here. This extension is important for scattered data approximation, where positive quadratures of high algebraic exactness may be unavailable.

The following lemma will be used repeatedly. The identity is algebraic and requires neither positivity of the weights nor exactness of $Q$.
\begin{lemma}\label{lem:hyperinterpolation-lemma}
Let $Q$ be a quadrature formula. For any $f\in C(\mathbb{S}^2)$ and
$p\in\mathbb{P}_n(\mathbb{S}^2)$, it holds that
\begin{equation*}
    \langle L_{n, Q}f,p\rangle_{L^2(\mathbb{S}^2)}=\langle f,p\rangle_Q.
\end{equation*}
\end{lemma}
\begin{proof}
    See \cite[Lemma~2.1]{an2026optimizationapproachweightcollocation}.
\end{proof}

\section{Existence of the \texorpdfstring{$L^1$--$L^2$}{L1-L2} Marcinkiewicz\texorpdfstring{--}{-}Zygmund Conditions} \label{sec:existence}
A primary goal of this paper is to establish the following $L^1$--$L^2$ MZ condition as the stability condition for hyperinterpolation. 

\begin{definition}[$L^1$--$L^2$ MZ Condition]
    Let $X = \{\mathbf{x}_1,\dots,\mathbf{x}_N\}\subset\mathbb{S}^2$ and $\mathbf{w}\in\mathbb{R}^N$. Let $Q$ be the quadrature formula associated with $X$ and $\mathbf{w}$. We say that $Q$ satisfies the \textit{$L^1$--$L^2$ MZ condition} of degree $n \in \mathbb{N}_0$ with constant $C > 0$ if 
    \begin{equation} \label{eq:L1-L2-MZ}
        \sum_{j=1}^N |w_j||p(\mathbf{x}_j)| \leq C \|p\|_{L^2(\mathbb{S}^2)} \quad \forall p\in\mathbb{P}_n(\mathbb{S}^2).
    \end{equation}
\end{definition}
This mixed-norm condition connects the discrete $L^1$ structure of the quadrature formula with the continuous $L^2$ geometry of the polynomial space. We next verify that it is natural to derive $L^1$--$L^2$ MZ conditions from the classical MZ conditions with established existence.

\subsection{Marcinkiewicz\texorpdfstring{--}{-}Zygmund Conditions}
Let us recall the MZ conditions initially introduced in \cite{gia2009localized}.

\begin{definition} \label{def:Marcinkiewicz-Zygmund-Conditions}
    Let $X = \{\mathbf{x}_1,\dots,\mathbf{x}_N\}\subset\mathbb{S}^2$ and $\mathbf{w}\in\mathbb{R}^N$. Let $Q$ be the quadrature formula associated with $X$ and $\mathbf{w}$. We say that $Q$ satisfies
    \begin{enumerate}[noitemsep]
        \item [(1)] the \textit{$L^1$ MZ condition} of degree $n \in \mathbb{N}_0$ with constant $c_1 > 0$ if
        \begin{equation} \label{eq:L1-MZ}
            \sum_{j=1}^N |w_j||p(\mathbf{x}_j)| \leq c_1 \|p\|_{L^1(\mathbb{S}^2)} \quad \forall\, p\in\mathbb{P}_n(\mathbb{S}^2);
        \end{equation}
        \item [(2)] the \textit{$L^2$ MZ condition} of degree $n \in \mathbb{N}_0$ with constant $c_2>0$ if
        \begin{equation} \label{eq:L2-MZ}
            \sum_{j=1}^N |w_j||p(\mathbf{x}_j)|^2 \le c_2 \|p\|_{L^2(\mathbb{S}^2)}^2 \quad \forall\, p\in\mathbb{P}_n(\mathbb{S}^2).
        \end{equation}
    \end{enumerate}
\end{definition}

Below we show that the $L^1$ and $L^2$ MZ condition separately imply the $L^1$--$L^2$ MZ condition \eqref{eq:L1-L2-MZ}.
\begin{lemma} \label{lem:L1-MZ-implies-L1-L2-MZ}
    A quadrature formula $Q$ satisfying the $L^1$ MZ condition of degree $n$ with constant $c_1 > 0$ also satisfies the $L^1$--$L^2$ MZ condition of degree $n$ with constant $C = \sqrt{4\pi} c_1$.
\end{lemma}
\begin{proof}
    By the Cauchy--Schwarz inequality, we have $\|p\|_{L^1(\mathbb{S}^2)} \leq \sqrt{4\pi}\|p\|_{L^2(\mathbb{S}^2)}$ for all $p \in \mathbb{P}_n(\mathbb{S}^2)$. Substituting this inequality into the $L^1$ MZ condition completes the proof.
\end{proof}

\begin{lemma} \label{lem:L2-MZ-implies-L1-L2-MZ}
    A quadrature formula $Q$ satisfying the $L^2$ MZ condition of degree $n$ with constant $c_2 > 0$ also satisfies the $L^1$--$L^2$ MZ condition of degree $n$ with constant $C = \sqrt{4\pi} c_2$.
\end{lemma}
\begin{proof}
    For any $p \in \mathbb{P}_n(\mathbb{S}^2)$, by the Cauchy--Schwarz inequality and the $L^2$ MZ condition, we have
    \begin{align*}
        \sum_{j=1}^N |w_j||p(\mathbf{x}_j)|
        &\leq \left(\sum_{j=1}^N |w_j|\right)^{1/2}
        \left(\sum_{j=1}^N |w_j||p(\mathbf{x}_j)|^2\right)^{1/2}\\
        &\leq \sqrt{c_2 \|e\|_{L^2(\mathbb{S}^2)}^2}\,
        \sqrt{c_2 \|p\|_{L^2(\mathbb{S}^2)}^2}
        = \sqrt{4\pi}c_2\|p\|_{L^2(\mathbb{S}^2)},
    \end{align*}
    and the proof is complete.
\end{proof}
Thus, according to \cref{lem:L1-MZ-implies-L1-L2-MZ,lem:L2-MZ-implies-L1-L2-MZ}, the existence of either the $L^1$ or $L^2$ MZ condition yields the existence of the $L^1$--$L^2$ MZ condition \eqref{eq:L1-L2-MZ}.

\subsection{From Classical MZ Theory to the \texorpdfstring{$L^1$--$L^2$}{L1-L2} MZ Condition}
\label{sec:classical-to-L2}
The study of MZ conditions on the sphere was initiated in the foundational paper \cite{mhaskar2001spherical}, where MZ inequalities were established for quadrature formulas associated with compatible partitions of the sphere. We now recall the relevant definitions.

\begin{definition}[\!\!\protect{\cite[Definition~3.1]{mhaskar2001spherical}}]
    A finite collection $\mathcal{R} := \{R_1, R_2, \ldots, R_N\}$ of closed, nonoverlapping (i.e., having no common interior points) subsets such that $\mathbb{S}^2 = \bigcup_{j=1}^N R_j$ is called a \textit{partition} of $\mathbb{S}^2$. We say that $\mathcal{R}$ is \textit{$X$-compatible} if each \textit{patch} $R_j \in \mathcal{R}$ contains exactly one point $\mathbf{x}_j \in X$ in its interior. The \textit{partition weight} of $\mathcal{R}$ is the vector $\mathbf{r} = (\omega(R_1), \omega(R_2), \ldots, \omega(R_N))^\top \in \mathbb{R}^N$. The \textit{partition norm} of $\mathcal{R}$ is defined by its largest patch diameter: $\|\mathcal{R}\| := \max_{R_j \in \mathcal{R}} \mathrm{diam}\, R_j$, where $\mathrm{diam}\, S := \sup_{\mathbf{x}, \mathbf{y} \in S} \mathrm{dist}(\mathbf{x}, \mathbf{y})$.
\end{definition}

Since quadrature formulas satisfying the $L^1$ (or $L^2$) MZ condition naturally satisfy the $L^1$--$L^2$ MZ condition (\ref{eq:L1-L2-MZ}), we obtain the following existence theorem.
\begin{theorem}[Existence of $L^1$--$L^2$ MZ Conditions]
\label{thm:existence-of-L1-L2-MZ}
Let $\mathcal{R}$ be an $X$-compatible partition of $\mathbb{S}^2$ with partition weight $\mathbf{r}$. For $n\geq 1$, if $\eta \in (0, 1)$ and the partition norm $\|\mathcal{R}\|$ satisfies
\begin{equation} \label{eq:geometric-condition}
    \|\mathcal{R}\| < \frac{\eta}{n} \varrho
\end{equation}
where $\varrho>0$ is a fixed constant, such as $\varrho_2$ defined in \cite[(3.25)]{mhaskar2001spherical} or the constant determined by $C_2$ defined in \cite[Theorem~4.1]{filbir2024marcinkiewicz}. Then, $Q_{X, \mathbf{r}}$ satisfies the $L^1$--$L^2$ MZ condition of degree $n$ with constant $C=\sqrt{4\pi}(1+\eta)$.
\end{theorem}

\begin{proof}
Under the condition \eqref{eq:geometric-condition}, the $L^1$ version of \cite[Theorem~3.1]{mhaskar2001spherical} asserts that $Q_{X, \mathbf{r}}$ satisfies the $L^1$ MZ condition of degree $n$ with constant $c_1 = 1+\eta$. It remains to use \cref{lem:L1-MZ-implies-L1-L2-MZ}.
\end{proof}

\begin{remark}
The same argument applies if one instead uses the $L^2$ version of \cite[Theorem~3.1]{mhaskar2001spherical} together with \cref{lem:L2-MZ-implies-L1-L2-MZ}. The resulting constant remains $C=\sqrt{4\pi}(1+\eta)$.
\end{remark}

\section{Quadrature Convergence Does Not Imply Hyperinterpolation Convergence} \label{sec:does-not-imply}

Classical convergence theory of numerical integration treats quadrature formulas as bounded linear functionals on $C(\mathbb{S}^2)$ and characterizes convergence by accuracy on a dense subspace together with uniform boundedness of the functional norms; see \cite{polya1933konvergenz} and the functional-analytic formulation in \cite[Theorem~12.4]{kress2014linear}. Since hyperinterpolation is obtained from the orthogonal projection by replacing the $L^2$ inner products with quadrature evaluations, it is intuitive to conjecture that the convergence of the underlying quadrature formulas is sufficient to ensure the $L^2$ convergence of their associated hyperinterpolation operators. In this section, however, we construct a convergent sequence of quadrature formulas whose associated hyperinterpolation operators do not converge in the $L^2$ sense. Thus, quadrature convergence alone does not imply the $L^2$ convergence of hyperinterpolation.

\subsection{P\texorpdfstring{\'o}{ó}lya's Conditions on Quadrature Convergence}

\begin{definition}[Uniform P\'olya Stability]
\label{def:polya-stability}
A sequence of quadrature formulas $\{Q_n\}_{n\in\mathbb{N}}$ is said to satisfy the \textit{uniform P\'olya stability condition} if there exists a constant $M>0$ independent of $n$ such that
\begin{equation}\label{eq:polya-cond}
    \sup_{n\in\mathbb{N}} \|\mathbf{w}^{(n)}\|_1
    \leq M.
\end{equation}
\end{definition}

\begin{remark}\label{rem:quadrature-stability}
By \eqref{eq:quadrature-functional-norm}, $\|Q_n\|_{C(\mathbb{S}^2) \to \mathbb{R}} = \|\mathbf{w}^{(n)}\|_1$. The substantive requirement in \eqref{eq:polya-cond} is therefore the uniform boundedness of the operator norms of the quadrature functionals
\begin{equation*}
    \sup_{n \in \mathbb{N}} \|Q_n\|_{C(\mathbb{S}^2) \to \mathbb{R}} \leq M.
\end{equation*}
\end{remark}

The uniform P\'olya stability condition implies that each $Q_n$ satisfies the $L^1$--$L^2$ MZ condition of degree $0$ with constant $M$. It gives no uniform control on higher degree polynomials. As will be shown in \cref{prop:Polya-does-not-imply-L1-L2-MZ}, this distinction is precisely why the uniform boundedness of the operator norms of quadrature functionals does not imply that of the associated hyperinterpolation operators.

\begin{lemma}\label{lem:L1-L2-MZ-implies-Polya}
Let $\{Q_n\}_{n \in \mathbb{N}}$ be a sequence of quadrature formulas. Suppose each $Q_n$ satisfies the $L^1$--$L^2$ MZ condition of degree $n$ with constant $C > 0$ independent of $n$, then $\|Q_n\|_{C(\mathbb{S}^2) \to \mathbb{R}} \leq \sqrt{4\pi} C$. Consequently, $\{Q_n\}_{n\in\mathbb{N}}$ satisfies the uniform P\'olya stability condition with constant $M=\sqrt{4\pi}C$.
\end{lemma}
\begin{proof}
Using the norm identity \eqref{eq:quadrature-functional-norm} and taking $e \in \mathbb{P}_0(\mathbb{S}^2)$ in the $L^1$--$L^2$ MZ condition give
\begin{equation*}
    \|Q_n\|_{C(\mathbb{S}^2) \to \mathbb{R}} = \sum_{j=1}^{N}|w_j^{(n)}| \leq C\|e\|_{L^2(\mathbb{S}^2)} = \sqrt{4\pi}C.
\end{equation*}
The proof is complete.
\end{proof}

Let us recall P\'olya's classical theorem on quadrature convergence \cite{polya1933konvergenz}, in its natural formulation on the unit sphere; see also \cite[Theorem~12.4]{kress2014linear}.
\begin{theorem}[P\'olya's Theorem on Quadrature Convergence
\protect{\cite{polya1933konvergenz}}]\label{thm:polya}
Let $\{Q_n\}_{n\in\mathbb{N}}$ be a sequence of quadrature formulas. Then, $\{Q_n\}_{n \in \mathbb{N}}$ converges for every continuous function on the sphere, that is,
\begin{equation*}
    \lim_{n\to\infty}|Q_n(f)-I(f)|=0
    \quad \forall\,f\in C(\mathbb{S}^2)
\end{equation*}
if and only if $\{Q_n\}_{n \in \mathbb{N}}$ satisfies both of the following conditions:
\begin{enumerate}[(1),noitemsep]
    \item the uniform P\'olya stability condition \eqref{eq:polya-cond}; and
    \item the asymptotic approximation property for polynomials,
    \begin{equation}\label{eq:approximation-property}
        \lim_{n\to\infty}|Q_n(p)-I(p)|=0
        \quad \forall\,p\in\mathbb{P}(\mathbb{S}^2).
    \end{equation}
\end{enumerate}
\end{theorem}
% \begin{proof}
% We first prove the necessity. Suppose $\lim_{n \to \infty} |Q_n(f) - I(f)| = 0$ for all $f\in C(\mathbb{S}^2)$. Then \eqref{eq:approximation-property} is immediate. Moreover, for every fixed $f\in C(\mathbb{S}^2)$, the sequence $\{Q_n(f)\}$ is bounded. The uniform boundedness principle therefore implies
% \begin{equation*}
%     \sup_{n\in\mathbb{N}}\|Q_n\|_{C\to\mathbb{R}}<\infty,
% \end{equation*}
% which is precisely \eqref{eq:polya-cond}.

% We next prove the sufficiency. Suppose \eqref{eq:approximation-property} and
% \eqref{eq:polya-cond} hold. Fix $f\in C(\mathbb{S}^2)$. Given any $\epsilon>0$, since
% $\mathbb{P}(\mathbb{S}^2)$ is dense in $C(\mathbb{S}^2)$, we can choose
% $p\in\mathbb{P}(\mathbb{S}^2)$ such that
% \begin{equation*}
%     \|f-p\|_\infty<\frac{\epsilon}{2(M+4\pi)}.
% \end{equation*}
% Since
% $\|I\|_{C\to\mathbb{R}}=I(1)=4\pi$, we obtain
% \begin{align*}
%     |Q_n(f)-I(f)|
%     &\leq |Q_n(f-p)|+|Q_n(p)-I(p)|+|I(p-f)| \\
%     &\leq (M+4\pi)\|f-p\|_\infty+|Q_n(p)-I(p)|
% \end{align*}
% For all sufficiently large $n$, \eqref{eq:approximation-property} gives
% $|Q_n(p)-I(p)|<\epsilon/2$, and hence $|Q_n(f) - I(f)| \leq \epsilon$. Thus $\lim_{n \to \infty} |Q_n(f) - I(f)| = 0$ for every $f\in C(\mathbb{S}^2)$.
% \end{proof}

The uniform P\'olya stability condition serves as a stability condition that guarantees the stability of quadrature functionals, while the asymptotic approximation property for polynomials is an accuracy condition on the dense subspace $\mathbb{P}(\mathbb{S}^2)$ of $C(\mathbb{S}^2)$.

\subsection{Divergence of Hyperinterpolation Under P\texorpdfstring{\'o}{ó}lya's Conditions}
We now construct a sequence of quadrature formulas satisfying P\'olya's conditions but for which $\|L_n\|_{C(\mathbb{S}^2) \to L^2(\mathbb{S}^2)}$ diverges. This will in turn imply the existence of a continuous function for which hyperinterpolation does not converge in $L^2$.

\begin{theorem} \label{thm:hyper-2-neg}
There exists a sequence $\{Q_n\}_{n\in\mathbb{N}}$ of quadrature formulas satisfying the uniform P\'olya stability condition \eqref{eq:polya-cond} and the asymptotic approximation property for polynomials \eqref{eq:approximation-property}, such that the associated degree-$n$ hyperinterpolation operators $L_n:=L_{n,Q_n}$ satisfy
\begin{equation*}
    \lim_{n\to\infty}\|L_n\|_{C(\mathbb{S}^2)\to L^2(\mathbb{S}^2)}=\infty.
\end{equation*}
Consequently, there exists $f\in C(\mathbb{S}^2)$ such that
$\limsup_{n\to\infty}\|L_nf\|_{L^2(\mathbb{S}^2)}=\infty$.
\end{theorem}
\begin{proof}
    Let $\{T^{(n)}\}_{n\in \mathbb{N}}$ be a sequence of spherical $n$-designs consisting of $t_{n}$ points. Applying rotations if necessary, we assume each $T^{(n)}$ does not contain the north pole $\mathbf{n} = (0,0,1)^{\top}$.
    
    The properties of $G_n$ in \cref{lem:Gn-properties} imply, for any $\mathbf{x} \in \mathbb{S}^2$, that
    \begin{equation*}
       \|G_n(\cdot, \mathbf{x}) - G_n(\cdot, \mathbf{n})\|_{L^2(\mathbb{S}^2)}^2 = G_n(\mathbf{x}, \mathbf{x}) - 2G_n(\mathbf{x}, \mathbf{n}) + G_n(\mathbf{n}, \mathbf{n}) = 2G_n(\mathbf{n}, \mathbf{n})-2G_n(\mathbf{x}, \mathbf{n}).
    \end{equation*}
    The continuity of $G_n$ implies that there exists $\rho_n>0$ such that
    \begin{equation*}
        \| G_n(\cdot,\mathbf{x}) - G_n(\cdot,\mathbf{n}) \|_{L^2(\mathbb{S}^2)} \leq 1
    \end{equation*}
    whenever $\mathrm{dist}(\mathbf{x}, \mathbf{n}) \leq \rho_n$.
   
    Let $s_n := (n+1)^2+1$. Choose distinct points $S^{(n)} :=\{\mathbf{x}_1^{(n)},\dots,\mathbf{x}_{s_n}^{(n)}\}$ inside a spherical cap $\{\mathbf{x}\in\mathbb{S}^2: \mathrm{dist}(\mathbf{x},\mathbf{n}) \leq \delta_n\} $ with $\delta_n := \min\{\mathrm{dist}(\mathbf{n}, T^{(n)})/2, \rho_n\}$. 
    Since $\dim\mathbb{P}_n(\mathbb{S}^2) = (n+1)^2<s_n$, the matrix 
    \begin{equation} \label{eq:basis-matrix}
        \mathbf{Y}_n := [Y_{\ell,k}(\mathbf{x}_j^{(n)})
        ]_{\substack{0\leq \ell\leq n,\; |k| \leq \ell\\
        1\leq j\leq s_n}}
        \in \mathbb{R}^{(n+1)^2\times s_n}
    \end{equation}
    has a nontrivial nullspace. Choose a nonzero vector $\mathbf{u}^{(n)}$ in the nullspace of $\mathbf{Y}_n$ and normalize it so that $\|\mathbf{u}^{(n)}\|_1=1$. The resulting quadrature formula $Q_{S^{(n)}, \mathbf{u}^{(n)}}$ integrates all $p \in \mathbb{P}_n(\mathbb{S}^2)$ to zero.

    Now, we define $Q_n$ as the sum of the equal-weight quadrature formula based on $T^{(n)}$ and the null quadrature formula $Q_{S^{(n)}, \mathbf{u}^{(n)}}$,
    \begin{equation*}
        Q_n(f) := \frac{4\pi}{t_n}\sum_{\mathbf{t}\in T^{(n)}} f(\mathbf{t}) + \sum_{j=1}^{s_n} u_j^{(n)} f(\mathbf{x}_j^{(n)}).
    \end{equation*}
    By construction, $Q_n$ is exact of degree $n$. Moreover, it holds that its weight $\mathbf{w}^{(n)}$ satisfies
    \begin{equation*}
        \|\mathbf{w}^{(n)}\|_1 = \sum_{\mathbf{t} \in T^{(n)}} \frac{4\pi}{t_n} + \sum_{j=1}^{s_n} |u_j^{(n)}| = 4\pi + 1.
    \end{equation*}
    Therefore, the sequence $\{Q_n\}_{n \in \mathbb{N}}$ satisfies both the uniform P\'olya stability condition \eqref{eq:polya-cond} with constant $M=4\pi+1$ and the asymptotic approximation property for polynomials \eqref{eq:approximation-property}.

    Since $X^{(n)} = T^{(n)} \cup S^{(n)}$ is a finite non-singleton set of distinct points, it follows that
    \begin{equation*}
        q_n := \min_{\substack{\mathbf{x}, \mathbf{y} \in X^{(n)} \\ \mathbf{x} \neq \mathbf{y}}} \mathrm{dist}(\mathbf{x}, \mathbf{y}) > 0.
    \end{equation*}
    For $j = 1, \ldots, s_n$, let us define the auxiliary functions $\psi_{j,n}(\mathbf{x}) := \max\{1- 3\mathrm{dist}(\mathbf{x}, \mathbf{x}_j^{(n)})/q_n, 0\}$. By construction, the supports of $\psi_{j, n}$ are pairwise disjoint. Consider now $f_n \in C(\mathbb{S}^2)$ given by
    \begin{equation*}
        f_n(\mathbf{x}) := \sum_{j=1}^{s_n} \mathrm{sgn}(u_j^{(n)})\psi_{j,n}(\mathbf{x}),
    \end{equation*}
    with the convention that $\mathrm{sgn}(0) = 0$. Then, $f_n(\mathbf{t})=0$ for every $\mathbf{t}\in T^{(n)}$, $f_n(\mathbf{x}_j^{(n)})=\mathrm{sgn}(u_j^{(n)})$ for $j=1,\ldots,s_n$, and $\|f_n\|_\infty=1$. Note that
    \begin{equation*}
        L_n f_n(\mathbf{x}) = \langle f_n, G_n(\mathbf{x}, \cdot)\rangle_{Q_n} = \sum_{j=1}^{s_n} |u_j^{(n)}|\, G_n(\mathbf{x},\mathbf{x}_j^{(n)}).
    \end{equation*}
    Since $\mathrm{dist}(\mathbf{x}_j^{(n)}, \mathbf{n}) \leq \rho_n$ and $\|\mathbf{u}^{(n)}\|_1 = 1$, the triangle inequality implies
    \begin{equation*}
        \|L_nf_n - G_n(\cdot,\mathbf{n})\|_{L^2(\mathbb{S}^2)} \leq \sum_{j=1}^{s_n} |u_j^{(n)}|\|G_n(\cdot,\mathbf{x}_j^{(n)}) - G_n(\cdot, \mathbf{n})\|_{L^2(\mathbb{S}^2)} \leq 1.
    \end{equation*}
    The reproducing property \eqref{eq:reproducing-property} implies $\|G_n(\cdot,\mathbf{n})\|_{L^2(\mathbb{S}^2)} = (n+1)/\sqrt{4\pi}$. Consequently,
    \begin{equation*}
        \|L_n f_n\|_{L^2(\mathbb{S}^2)} \geq \|G_n(\cdot,\mathbf{n})\|_{L^2(\mathbb{S}^2)} - \|L_n f_n - G_n(\cdot,\mathbf{n})\|_{L^2(\mathbb{S}^2)} \geq \frac{n+1}{\sqrt{4\pi}} - 1.
    \end{equation*}
    Hence, $\|L_n\|_{C(\mathbb{S}^2)\to L^2(\mathbb{S}^2)} \geq \|L_n f_n\|_{L^2(\mathbb{S}^2)} / \|f_n\|_\infty = (n+1)/\sqrt{4\pi} - 1 \to \infty$ as $n\to\infty$. By the uniform boundedness principle, there exists $f \in C(\mathbb{S}^2)$ such that $\limsup_{n\to\infty}\|L_n f\|_{L^2(\mathbb{S}^2)} = \infty$, and the proof is complete.
\end{proof}

\subsection{Limitations of P\texorpdfstring{\'o}{ó}lya's Conditions}
The divergence established in \cref{thm:hyper-2-neg} shows that P\'olya's two conditions for quadrature convergence are insufficient for hyperinterpolation in two distinct senses: the stability condition \eqref{eq:polya-cond} does not control polynomials of nonzero degrees, while the accuracy condition \eqref{eq:approximation-property} does not even guarantee convergence of the hyperinterpolants of the constant function.

\begin{proposition} \label{prop:Polya-does-not-imply-L1-L2-MZ}
There exists a sequence $\{Q_n\}_{n\in\mathbb{N}}$ of quadrature formulas satisfying the uniform P\'olya stability condition \eqref{eq:polya-cond}, whereas there exists no constant $C>0$ independent of $n$ such that every $Q_n$ satisfies the $L^1$--$L^2$ MZ condition of degree $n$ with constant $C$.
\end{proposition}

\begin{proof}
Let $\{X^{(n)}\}_{n\in\mathbb{N}}$ be a sequence of spherical $n$-designs consisting of $N_n$ points. Applying rotations if necessary, we assume that each $X^{(n)}$ contains the north pole $\mathbf{n}=(0,0,1)^\top$ as its first node, i.e., $\mathbf{x}_1^{(n)}=\mathbf{n}$ for each $n\in\mathbb{N}$. We define $Q_n$ on the nodes $X^{(n)}$ by assigning the Dirac-like weights
\begin{equation*}
w_j^{(n)} = \left\{\begin{array}{ll}
1, & \text{if } j = 1, \\ 
0, & \text{if } j \neq 1. 
\end{array}\right.
\end{equation*}
By construction,
\begin{equation*}
\|\mathbf{w}^{(n)}\|_1 = \sum_{j=1}^{N_n} |w_j^{(n)}| = 1 \quad \forall \, n \in \mathbb{N},
\end{equation*}
so the sequence $\{Q_n\}_{n \in \mathbb{N}}$ satisfies the uniform P\'olya stability condition \eqref{eq:polya-cond} with constant $M=1$. 

Define the auxiliary polynomial $\phi_n$ by
\begin{equation} \label{eq:auxiliary-kernel}
\phi_n(\mathbf{x}) := \frac{G_n(\mathbf{x}, \mathbf{n})}{\sqrt{G_n(\mathbf{n}, \mathbf{n})}} \in \mathbb{P}_n(\mathbb{S}^2) .
\end{equation}
Using the reproducing property \eqref{eq:reproducing-property}, we have $\|\phi_n\|_{L^2(\mathbb{S}^2)} = 1$. Evaluating the discrete $L^1$-norm of $\phi_n$ associated with the quadrature formula $Q_n$ yields
\begin{equation*}
\sum_{j=1}^{N_n} |w_j^{(n)}| |\phi_n(\mathbf{x}_j^{(n)})| = |w_1^{(n)}|\phi_n(\mathbf{n}) = \sqrt{G_n(\mathbf{n}, \mathbf{n})} = \frac{n+1}{\sqrt{4\pi}}.
\end{equation*}
Now, observe that
\begin{equation*}
\frac{\sum_{j=1}^{N _n} |w_j^{(n)}| |\phi_n(\mathbf{x}_j^{(n)})|}{\|\phi_n\|_{L^2(\mathbb{S}^2)}} = \frac{n+1}{\sqrt{4\pi}} \to \infty \quad \text{as } n \to \infty.
\end{equation*}
Consequently, no constant $C$ exists such that every $Q_n$ satisfies the $L^1$--$L^2$ MZ condition of degree $n$ with constant $C$.
\end{proof}

\begin{proposition} \label{prop:AP-does-not-imply-FAP}
There exists a sequence $\{Q_n\}_{n\in\mathbb{N}}$ of quadrature formulas satisfying the asymptotic approximation property for polynomials \eqref{eq:approximation-property}, but the associated degree-$n$ hyperinterpolation operators $L_n := L_{n, Q_n}$ satisfy $\|L_n e - e\|_{L^2(\mathbb{S}^2)} \equiv 1$.
\end{proposition}

\begin{proof}
Let $\{X^{(n)}\}_{n\in\mathbb{N}}$ be a sequence of spherical $2n$-designs consisting of $N_n$ points, and $v_j^{(n)} := 4\pi/N_n$. We construct each quadrature formula $Q_n$ using nodes $X^{(n)}$ and perturbed weights $w_j^{(n)} = v_j^{(n)} (1 + Y_{n,0}(\mathbf{x}_j^{(n)}))$. Each $Q_n$ is exact of degree $n-1$. Indeed, for any $p\in\mathbb{P}_{n-1}(\mathbb{S}^2)$, we have
\begin{equation*}
Q_n(p) = \sum_{j=1}^{N_n} v_j^{(n)} p(\mathbf{x}_j^{(n)}) + \sum_{j=1}^{N_n} v_j^{(n)} Y_{n,0}(\mathbf{x}_j^{(n)}) p(\mathbf{x}_j^{(n)}).
\end{equation*}
Since $p, Y_{n,0} p \in \mathbb{P}_{2n}(\mathbb{S}^2)$ and $X^{(n)}$ is a spherical $2n$-design, both sums are exact. Thus, we have $Q_n(p) = I(p) + I(Y_{n,0} p)$. Orthonormality of spherical harmonics implies $I(Y_{n,0} p) = 0$. Therefore, $Q_n(p) = I(p)$ for all $p \in \mathbb{P}_{n-1}(\mathbb{S}^2)$, which immediately implies the asymptotic approximation property for polynomials \eqref{eq:approximation-property}.

By definition, we have
\begin{equation*}
L_n e(\mathbf{x}) = \sum_{\ell=0}^n \sum_{k=-\ell}^\ell \langle e, Y_{\ell,k}\rangle_{Q_n} Y_{\ell,k}(\mathbf{x}).
\end{equation*}
For any $(\ell, k)$ with $0 \le \ell \le n$ and $|k| \leq \ell$, observe that $\left(1 + Y_{n,0}\right) Y_{\ell,k} \in \mathbb{P}_{2n}(\mathbb{S}^2)$. Since $X^{(n)}$ is a spherical $2n$-design, this means
\begin{align*}
\langle e, Y_{\ell,k}\rangle_{Q_n} & = \sum_{j=1}^{N_n} v_j^{(n)}(1+Y_{n,0}(\mathbf{x}_j^{(n)}))Y_{\ell,k}(\mathbf{x}_j^{(n)}) \\
& = I(Y_{\ell,k}) + I(Y_{n,0} Y_{\ell,k}) = \sqrt{4\pi}\langle Y_{0,0}, Y_{\ell,k}\rangle_{L^2(\mathbb{S}^2)} + \langle Y_{n,0}, Y_{\ell,k}\rangle_{L^2(\mathbb{S}^2)}.
\end{align*}
Using the orthonormality of spherical harmonics, we obtain
\begin{equation*}
\langle e, Y_{\ell,k}\rangle_{Q_n} = \begin{cases}
\sqrt{4\pi}, & \text{if } \ell = 0, k = 0, \\
1, & \text{if } \ell = n, k = 0, \\
0, & \text{otherwise}.
\end{cases}
\end{equation*}
Reconstructing the hyperinterpolant gives
$L_n e(\mathbf{x}) = \sqrt{4\pi}Y_{0,0}(\mathbf{x}) + Y_{n,0}(\mathbf{x}) = e(\mathbf{x}) + Y_{n,0}(\mathbf{x})$. Therefore
$\|L_n e - e\|_{L^2(\mathbb{S}^2)} = \|Y_{n,0}\|_{L^2(\mathbb{S}^2)} = 1$ for all $n \in \mathbb{N}$, which does not converge to $0$ as $n \to \infty$.
\end{proof}

\section{\texorpdfstring{$L^2$}{L2} Convergence of Hyperinterpolation}
\label{sec:equivalence}
Motivated by P\'olya's theorem in \cite{polya1933konvergenz}, we seek a stability--accuracy characterization of the $L^2$ convergence of hyperinterpolation. The stability condition is provided by the $L^1$--$L^2$ MZ condition \eqref{eq:L1-L2-MZ}, while the corresponding accuracy condition is a strengthening of \eqref{eq:approximation-property} as below.

\subsection{Accuracy Condition for Hyperinterpolation}
\begin{proposition}\label{prop:fap-implies-ap}
    Let $\{Q_n\}_{n\in\mathbb{N}}$ be a sequence of quadrature formulas and $L_n:=L_{n,Q_n}$ denote the associated degree-$n$ hyperinterpolation operator. If $\{Q_n\}_{n \in \mathbb{N}}$ satisfies the asymptotic functional approximation property for polynomials
    \begin{equation} \label{eq:functional-approximation-property}
        \lim_{n \to \infty} \|L_np - p\|_{L^2(\mathbb{S}^2)} = 0 \quad \forall\, p \in \mathbb{P}(\mathbb{S}^2),
    \end{equation}
    then $\{Q_n\}_{n \in \mathbb{N}}$ also satisfies the asymptotic approximation property for polynomials \eqref{eq:approximation-property}.
\end{proposition}
\begin{proof}
    Let $p \in \mathbb{P}(\mathbb{S}^2)$ be an arbitrary spherical polynomial. Note that $Q_n(p) = \langle p, e \rangle_{Q_n}$ for any $n \in \mathbb{N}$ and $I(p) = \langle p, e \rangle_{L^2(\mathbb{S}^2)}$. As $e \in \mathbb{P}_n(\mathbb{S}^2)$ for any $n \in \mathbb{N}$, \cref{lem:hyperinterpolation-lemma} applies and we have $\langle L_np, e\rangle_{L^2(\mathbb{S}^2)} = \langle p, e \rangle_{Q_n}$. Now, using this identity and the Cauchy--Schwarz inequality, we have
    \begin{align*}
        |Q_n(p) - I(p)| & = |\langle p, e \rangle_{Q_n} - \langle p,  e\rangle_{L^2(\mathbb{S}^2)}| = |\langle L_np, e \rangle_{L^2(\mathbb{S}^2)} - \langle p,  e\rangle_{L^2(\mathbb{S}^2)}| \notag \\
        & \leq \|L_np - p\|_{L^2(\mathbb{S}^2)}\|e\|_{L^2(\mathbb{S}^2)} = \sqrt{4\pi} \|L_np - p\|_{L^2(\mathbb{S}^2)}.
    \end{align*}
    It follows from \eqref{eq:functional-approximation-property} that
    \begin{equation*}
        \lim_{n \to \infty} |Q_n(p) - I(p)| \leq \lim_{n \to \infty} \sqrt{4\pi} \|L_np - p\|_{L^2(\mathbb{S}^2)} = 0,
    \end{equation*}
    which implies that $\{Q_n\}_{n \in \mathbb{N}}$ satisfies the asymptotic approximation property for polynomials \eqref{eq:approximation-property}.
\end{proof}

\begin{remark} \label{rem:satisfaction}
    The asymptotic functional approximation property for polynomials \eqref{eq:functional-approximation-property} can be ensured in several ways. A classical sufficient condition is that each quadrature formula $Q_n$ be exact of degree $2n$. In this case, $L_np = p$ for all $p \in \mathbb{P}_n(\mathbb{S}^2)$, and hence \eqref{eq:functional-approximation-property} follows immediately. The same property may also be obtained from $L^2$ MZ inequalities. For example, suppose that each $Q_n$ satisfies the $L^2$ MZ inequality of degree $n$ with constant $\eta_n > 0$ \cite{mhaskar2001spherical},
    \begin{equation*}
        |Q_n(p^2) - I(p^2)| \leq \eta_n I(p^2) \quad \forall\, p \in \mathbb{P}_n(\mathbb{S}^2).
    \end{equation*}
    By \cite[Proposition~5.4]{an2026optimizationapproachweightcollocation}, this condition implies
    \begin{equation*}
        \|L_np - p\|_{L^2(\mathbb{S}^2)} \leq \eta_n \|p\|_{L^2(\mathbb{S}^2)} \quad \forall\, p \in \mathbb{P}_n(\mathbb{S}^2).
    \end{equation*}
    Consequently, if $\eta_n \to 0$ as $n \to \infty$, then $\{Q_n\}_{n \in \mathbb{N}}$ also satisfies \eqref{eq:functional-approximation-property}.
\end{remark}

\subsection{Stability Condition for Hyperinterpolation}
We first establish an exact equivalence between the $L^1$--$L^2$ MZ bound and the operator-norm bound for hyperinterpolation. As a consequence, the optimal $L^1$--$L^2$ MZ constant is identified with the operator norm of the associated hyperinterpolation operator.

\begin{theorem}
\label{thm:L1-L2-MZ-equivalent-to-hyperinterpolation-stability}
Let $Q$ be a quadrature formula and $L_n:=L_{n,Q}$ denote the associated degree-$n$ hyperinterpolation operator. Then, for any fixed constant $C > 0$, $Q$ satisfies the $L^1$--$L^2$ MZ condition of degree $n$ with constant $C$ if and only if
\begin{equation*}
    \|L_n\|_{C(\mathbb{S}^2)\to L^2(\mathbb{S}^2)}\leq C.
\end{equation*}
\end{theorem}

\begin{proof}
Suppose first that $Q$ satisfies the $L^1$--$L^2$ MZ condition of degree $n$ with constant $C$. For any $f\in C(\mathbb{S}^2)$, we have $L_nf\in\mathbb{P}_n(\mathbb{S}^2)$. Applying \cref{lem:hyperinterpolation-lemma} with $p=L_nf$, followed by the $L^1$--$L^2$ MZ condition, gives
\begin{equation*}
    \|L_nf\|_{L^2(\mathbb{S}^2)}^2 = \langle L_nf,L_nf\rangle_{L^2(\mathbb{S}^2)} = \langle f,L_nf\rangle_Q \leq \|f\|_\infty\sum_{j=1}^{N}|w_j||L_nf(\mathbf{x}_j)| \leq C\|f\|_\infty\|L_nf\|_{L^2(\mathbb{S}^2)}.
\end{equation*}
Consequently, $\|L_nf\|_{L^2(\mathbb{S}^2)}\leq C\|f\|_\infty$, and hence $\|L_n\|_{C(\mathbb{S}^2)\to L^2(\mathbb{S}^2)}\leq C$.

Conversely, suppose that $\|L_n\|_{C(\mathbb{S}^2)\to L^2(\mathbb{S}^2)}\leq C$. Fix $p\in\mathbb{P}_n(\mathbb{S}^2)$. Set
\begin{equation*}
    q:=
    \begin{cases}
    \min_{\substack{ \mathbf{x}_i, \mathbf{x}_j \in X \\ \mathbf{x}_i \neq \mathbf{x}_j}}
    \mathrm{dist}(\mathbf{x}_i,\mathbf{x}_j),&N\geq2,\\[1.2ex]
    1,&N=1.
    \end{cases}
\end{equation*}
For $j=1,\ldots,N$, let us define the auxiliary functions $\psi_j(\mathbf{x}):=\max\{1-3\mathrm{dist}(\mathbf{x},\mathbf{x}_j)/q,0\}$. By construction, the supports of $\psi_j$ are pairwise disjoint. Consider now $f_p \in C(\mathbb{S}^2)$ given by
\begin{equation*}
f_p(\mathbf{x}):=\sum_{j=1}^{N}\mathrm{sgn}(w_jp(\mathbf{x}_j))\psi_j(\mathbf{x}),
\end{equation*}
with the convention that $\mathrm{sgn}(0)=0$. Then, $\|f_p\|_\infty\leq 1$ and $f_p(\mathbf{x}_j) = \mathrm{sgn}(w_jp(\mathbf{x}_j))$ for $j = 1, \ldots, N$. Therefore, using \cref{lem:hyperinterpolation-lemma}, the Cauchy--Schwarz inequality, and the assumed operator-norm bound, we obtain
\begin{multline*}
    \sum_{j=1}^{N}|w_j||p(\mathbf{x}_j)|
    =\langle f_p,p\rangle_Q
     =\langle L_nf_p,p\rangle_{L^2(\mathbb{S}^2)} \\
    \leq \|L_nf_p\|_{L^2(\mathbb{S}^2)}
           \|p\|_{L^2(\mathbb{S}^2)}\leq C\|f_p\|_\infty\|p\|_{L^2(\mathbb{S}^2)}
     \leq C\|p\|_{L^2(\mathbb{S}^2)}.
\end{multline*}
Thus, $Q$ satisfies the $L^1$--$L^2$ MZ condition of degree $n$ with constant $C$.
\end{proof}

An immediate but important consequence of \cref{thm:L1-L2-MZ-equivalent-to-hyperinterpolation-stability} is the following sharp norm identity, which also motivates the duality interpretation developed in the next subsection.

\begin{corollary}\label{cor:sharp-MZ-constant}
The optimal constant in the $L^1$--$L^2$ MZ condition of degree $n$ for a quadrature formula $Q$ is exactly the operator norm of $L_n$. More precisely,
\begin{equation}\label{eq:sharp-MZ-constant}
    \|L_n\|_{C(\mathbb{S}^2)\to L^2(\mathbb{S}^2)}=\sup_{\substack{p\in\mathbb{P}_n(\mathbb{S}^2)\\p\neq0}}
      \frac{\sum_{j=1}^{N}|w_j||p(\mathbf{x}_j)|}{\|p\|_{L^2(\mathbb{S}^2)}}.
\end{equation}
\end{corollary}
\begin{proof}
This follows immediately from \cref{thm:L1-L2-MZ-equivalent-to-hyperinterpolation-stability} by taking the infimum over all admissible constants $C$.
\end{proof}

\subsection{A Banach Space Duality Interpretation} \label{sec:duality}
The norm identity in \cref{cor:sharp-MZ-constant} has a natural interpretation through Banach space duality. Let $Q$ be a quadrature formula and $L_n := L_{n,Q}$ denote the associated degree-$n$ hyperinterpolation operator. Consider
\begin{equation*}
    L_n:C(\mathbb{S}^2)\to\mathbb{P}_n(\mathbb{S}^2),
\end{equation*}
where $C(\mathbb{S}^2)$ is equipped with the uniform norm and $\mathbb{P}_n(\mathbb{S}^2)$ is equipped with the norm inherited from $L^2(\mathbb{S}^2)$. Its Banach adjoint is
\begin{equation*}
    L_n^*: \mathbb{P}_n(\mathbb{S}^2)^* \to C(\mathbb{S}^2)^*.
\end{equation*}
Since $\mathbb{P}_n(\mathbb{S}^2)$ is a finite-dimensional Hilbert space, the Riesz representation theorem identifies $\mathbb{P}_n(\mathbb{S}^2)^*$ isometrically with $\mathbb{P}_n(\mathbb{S}^2)$ through the $L^2$ inner product. Moreover, the Riesz--Markov--Kakutani representation theorem \cite[Theorem~7.17]{folland1999real} identifies $C(\mathbb{S}^2)^*$ isometrically with the Banach space $M(\mathbb{S}^2)$ of finite signed regular Borel measures on $\mathbb{S}^2$, endowed with the total variation norm
\begin{equation*}
    \|\mu\|_{M} := \sup_{\|f\|_\infty \leq 1}
    \left| \int_{\mathbb{S}^2} f(\mathbf{x})\,d\mu(\mathbf{x}) \right|.
\end{equation*}
Under these canonical isometric identifications, suppressing the Riesz isomorphism from the notation, we may regard the adjoint as an operator
\begin{equation*}
    L_n^*:\mathbb{P}_n(\mathbb{S}^2) \to M(\mathbb{S}^2),
\end{equation*}
mapping $p\in\mathbb{P}_n(\mathbb{S}^2)$ to the atomic signed measure
\begin{equation*}
    L_n^*p=\sum_{j=1}^{N}w_jp(\mathbf{x}_j)\delta_{\mathbf{x}_j}.
\end{equation*}
Indeed, for every $f\in C(\mathbb{S}^2)$, \cref{lem:hyperinterpolation-lemma} implies
\begin{equation*}
    (L_n^*p)(f) = \sum_{j=1}^{N}w_jp(\mathbf{x}_j)f(\mathbf{x}_j) =\langle f,p\rangle_Q = \langle L_nf,p\rangle_{L^2(\mathbb{S}^2)}.
\end{equation*}
Because the quadrature nodes are pairwise distinct, we have
\begin{equation*}
    \|L_n^*p\|_{M(\mathbb{S}^2)}
    =\sum_{j=1}^{N}|w_j||p(\mathbf{x}_j)|.
\end{equation*}
It follows that
\begin{equation*}
    \|L_n^*\|_{L^2(\mathbb{S}^2)\to M(\mathbb{S}^2)} = \sup_{\substack{p \in \mathbb{P}_n(\mathbb{S}^2) \\ p \neq 0}} \frac{\sum_{j=1}^N |w_j||p(\mathbf{x}_j)|}{\|p\|_{L^2(\mathbb{S}^2)}}.
\end{equation*}
Because a bounded linear operator and its Banach adjoint have the same operator norm, it holds that
\begin{equation*}
    \|L_n\|_{C(\mathbb{S}^2) \to L^2(\mathbb{S}^2)} = \|L_n^*\|_{L^2(\mathbb{S}^2)\to M(\mathbb{S}^2)},
\end{equation*}
in which \cref{eq:sharp-MZ-constant} is recovered.

This dual formulation shows that the $L^1$--$L^2$ MZ condition \eqref{eq:L1-L2-MZ} is intrinsic to the functional-analytic structure of hyperinterpolation. Its discrete $L^1$ expression is exactly the total variation norm of the atomic signed measure $L_n^*p$, whereas its continuous $L^2$ norm is inherited from the Riesz identification of $\mathbb{P}_n(\mathbb{S}^2)^*$ with $\mathbb{P}_n(\mathbb{S}^2)$. Thus, the mixed $L^1$--$L^2$ structure is intrinsically determined by the Banach adjoint of the hyperinterpolation operator rather than introduced ad hoc.

\subsection{An If-and-Only-If Characterization of the \texorpdfstring{$L^2$}{L2} Convergence of Hyperinterpolation}
The main theorem of this paper establishes the necessary and sufficient conditions for the $L^2$ convergence of hyperinterpolation.

\begin{theorem}\label{thm:convergence}
Let $\{Q_n\}_{n\in\mathbb{N}}$ be a sequence of quadrature formulas and $L_n:=L_{n,Q_n}$ denote the associated degree-$n$ hyperinterpolation operator. Then
\begin{equation*}
    \lim_{n\to\infty}\|L_nf-f\|_{L^2(\mathbb{S}^2)}=0
    \qquad\forall f\in C(\mathbb{S}^2)
\end{equation*}
if and only if the following two conditions hold:
\begin{enumerate}[(1), noitemsep]
    \item there exists a constant $C>0$, independent of $n$, such that every $Q_n$ satisfies the $L^1$--$L^2$ MZ condition \eqref{eq:L1-L2-MZ} of degree $n$ with constant $C$;
    \item $\{Q_n\}_{n \in \mathbb{N}}$ satisfies the asymptotic functional approximation property for polynomials \eqref{eq:functional-approximation-property}.
\end{enumerate}
\end{theorem}

\begin{proof}
We first prove necessity. Suppose that $\lim_{n\to\infty}\|L_nf-f\|_{L^2(\mathbb{S}^2)}=0$ for all $f\in C(\mathbb{S}^2)$. Then, \eqref{eq:functional-approximation-property} follows immediately because $\mathbb{P}(\mathbb{S}^2)\subseteq C(\mathbb{S}^2)$. For every fixed $f\in C(\mathbb{S}^2)$,
\begin{equation*}
    \|L_nf\|_{L^2(\mathbb{S}^2)}
    \leq \|L_nf-f\|_{L^2(\mathbb{S}^2)}+\|f\|_{L^2(\mathbb{S}^2)},
\end{equation*}
so the sequence $\{\|L_nf\|_{L^2(\mathbb{S}^2)}\}_{n\in\mathbb{N}}$ is bounded. Since $C(\mathbb{S}^2)$ is a Banach space, the uniform boundedness principle yields the existence of constant $C > 0$ such that
\begin{equation*}
\sup_{n\in\mathbb{N}}\|L_n\|_{C(\mathbb{S}^2)\to L^2(\mathbb{S}^2)} \leq C.
\end{equation*}
By \cref{thm:L1-L2-MZ-equivalent-to-hyperinterpolation-stability}, each $Q_n$ satisfies the $L^1$--$L^2$ MZ condition of degree $n$ with constant $C$. This establishes necessity.

We next prove sufficiency. Fix $f\in C(\mathbb{S}^2)$ and let $\epsilon>0$. Since $\mathbb{P}(\mathbb{S}^2)$ is dense in $C(\mathbb{S}^2)$, choose $p\in\mathbb{P}(\mathbb{S}^2)$ such that $(C+\sqrt{4\pi})\|f-p\|_\infty< \epsilon/2$. For every $n\in\mathbb{N}$, we have
\begin{align*}
    \|L_nf-f\|_{L^2(\mathbb{S}^2)}
    &\leq \|L_n(f-p)\|_{L^2(\mathbb{S}^2)}
      +\|L_np-p\|_{L^2(\mathbb{S}^2)}
      +\|p-f\|_{L^2(\mathbb{S}^2)} \\
    &\leq (C+\sqrt{4\pi})\|f-p\|_\infty
      +\|L_np-p\|_{L^2(\mathbb{S}^2)}.
\end{align*}
By \eqref{eq:functional-approximation-property}, there exists $n_0\in\mathbb{N}$ such that $\|L_np-p\|_{L^2(\mathbb{S}^2)}<\epsilon/2$ for all $n \geq n_0$. Therefore, $\|L_nf-f\|_{L^2(\mathbb{S}^2)}<\epsilon$ for all $n\geq n_0$. Since $\epsilon>0$ is arbitrary, it follows that $\lim_{n\to\infty}\|L_nf-f\|_{L^2(\mathbb{S}^2)}=0$, establishing the sufficiency.
\end{proof}

\section{Strict Logical Hierarchy of Conditions}
\label{sec:hierarchy}
We now provide a comprehensive summary of the logical relationships among various stability and accuracy conditions pertinent to quadrature and hyperinterpolation convergence.

\begin{theorem} \label{thm:hierarchy}
For a sequence of quadrature formulas $\{Q_n\}_{n \in \mathbb{N}}$ and the sequence of their associated hyperinterpolation operators $\{L_n\}_{n \in \mathbb{N}}$ (where $L_n := L_{n, Q_n}$), the stability and accuracy conditions governing quadrature convergence and the $L^2$ convergence of hyperinterpolation form the complete logical hierarchy shown in \cref{fig:hierarchy}.
\end{theorem}
\begin{proof}
Most implications, non-implications, and equivalences have already been established; the corresponding references are indicated in \cref{fig:hierarchy}. It remains to construct a sequence of quadrature formulas $\{Q_n\}_{n\in\mathbb{N}}$ that satisfies the $L^1$--$L^2$ MZ condition with constant independent of $n$, but does not satisfy either the $L^1$ or the $L^2$ MZ condition with constant independent of $n$. A construction similar to the one used in \cref{prop:AP-does-not-imply-FAP} serves as this purpose.

Let $\{X^{(n)}\}_{n\in\mathbb{N}}$ be a sequence of spherical $(2n+2)$-designs consisting of $N_n$ points, where $N_n\leq\varsigma(2n+2)^2$ for some constant $\varsigma>0$ independent of $n$ \cite{bondarenko2013optimal}. Applying rotations if necessary, assume that each $X^{(n)}$ contains the north pole $\mathbf{n}=(0,0,1)^\top$ as its first node, i.e., $\mathbf{x}_1^{(n)}=\mathbf{n}$. Set $v_j^{(n)}:=4\pi/N_n$ and define the quadrature formula $Q_n$ on $X^{(n)}$ with perturbed weights 
$$
w_j^{(n)} := v_j^{(n)}(1+Y_{n+1,0}(\mathbf{x}_j^{(n)})).
$$

We first show that $Q_n$ satisfies the $L^1$--$L^2$ MZ condition of degree $n$. For each $p \in \mathbb{P}_n(\mathbb{S}^2)$, we have
\begin{equation*}
    \sum_{j=1}^{N_n} |w_j^{(n)}||p(\mathbf{x}_j^{(n)})| \leq \sum_{j=1}^{N_n} v_j^{(n)}|p(\mathbf{x}_j^{(n)})| + \sum_{j=1}^{N_n} v_j^{(n)}|Y_{n+1,0}(\mathbf{x}_j^{(n)})||p(\mathbf{x}_j^{(n)})|.
\end{equation*}
By the Cauchy--Schwarz inequality, we can bound the first term of the right-hand side as
\begin{equation*}
    \sum_{j=1}^{N_n} v_j^{(n)}|p(\mathbf{x}_j^{(n)})| \leq \left(\sum_{j=1}^{N_n} v_j^{(n)}\right)^{1/2} \left(\sum_{j=1}^{N_n} v_j^{(n)}|p(\mathbf{x}_j^{(n)})|^2\right)^{1/2} = \|e\|_{L^2(\mathbb{S}^2)} \|p\|_{L^2(\mathbb{S}^2)} = \sqrt{4\pi}\|p\|_{L^2(\mathbb{S}^2)},
\end{equation*}
where the equality is due to the fact that $X^{(n)}$ is a spherical $(2n+2)$-design and $e, p^2 \in \mathbb{P}_{2n}(\mathbb{S}^2)$. Similarly, since $Y_{n+1,0}^2, p^2 \in \mathbb{P}_{2n+2}(\mathbb{S}^2)$, we can bound the second term as
\begin{align*}
    \sum_{j=1}^{N_n} v_j^{(n)}|Y_{n+1,0}(\mathbf{x}_j^{(n)})||p(\mathbf{x}_j^{(n)})|
    & \leq \left(\sum_{j=1}^{N_n} v_j^{(n)}|Y_{n+1,0}(\mathbf{x}_j^{(n)})|^2\right)^{1/2} \left(\sum_{j=1}^{N_n} v_j^{(n)}|p(\mathbf{x}_j^{(n)})|^2\right)^{1/2} \\
    & = \|Y_{n+1,0}\|_{L^2(\mathbb{S}^2)}\|p\|_{L^2(\mathbb{S}^2)} = \|p\|_{L^2(\mathbb{S}^2)}.
\end{align*}
Substituting these estimates yields
\begin{equation*}
    \sum_{j=1}^{N_n} |w_j^{(n)}||p(\mathbf{x}_j^{(n)})| \leq (\sqrt{4\pi} + 1)\|p\|_{L^2(\mathbb{S}^2)}.
\end{equation*}
Thus, each quadrature formula $Q_n$ satisfies the $L^1$--$L^2$ MZ condition of degree $n$ with constant $C = \sqrt{4\pi}+1$ independent of $n$.
    
Now we investigate the constant in the $L^1$ MZ condition for each $Q_n$. Define $\psi_n$ as the auxiliary polynomial given by
\begin{equation*}
    \psi_n(\mathbf{x}) := \phi_{m}^2(\mathbf{x}) \in \mathbb{P}_n(\mathbb{S}^2) \quad \text{with} \quad m := \lfloor n/2 \rfloor,
\end{equation*}
where the function $\phi_m$ is as defined in \eqref{eq:auxiliary-kernel}. Recall that $\|\phi_m\|_{L^2(\mathbb{S}^2)} = 1$, from which it follows that $\|\psi_n\|_{L^1} = 1$. Evaluating the discrete $L^1$-norm of $\psi_n$ associated with quadrature formula $Q_n$ yields
\begin{equation*}
\sum_{j=1}^{N_n} |w_j^{(n)}| |\psi_n(\mathbf{x}_j^{(n)})| \geq |w_1^{(n)}| |\psi_n(\mathbf{n})| = |v_1^{(n)}||1 + Y_{n+1,0}(\mathbf{n})|\phi_m^2(\mathbf{n}).
\end{equation*}
Recalling also that
\begin{equation*}
    v_1^{(n)} = \frac{4\pi}{N_n} \geq \frac{4\pi}{\varsigma (2n+2)^2}, \quad  Y_{n+1,0}(\mathbf{n}) = \sqrt{\frac{2n+3}{4\pi}}, \quad \text{and} \quad \phi_m(\mathbf{n}) = \frac{m+1}{\sqrt{4\pi}},
\end{equation*}
we obtain
\begin{align} 
\frac{\sum_{j=1}^{N_n} |w_j^{(n)}| |\psi_n(\mathbf{x}_j^{(n)})|}{\|\psi_n\|_{L^1}}
&\geq \frac{4\pi}{\varsigma(2n+2)^2}
\left(1+\sqrt{\frac{2n+3}{4\pi}}\right)
\frac{(\lfloor n/2\rfloor+1)^2}{4\pi} \notag\\
& \gtrsim 1+\sqrt{n}. \label{eq:asymptotic-lower-bound}
\end{align}
Thus, it is impossible for each $Q_n$ to satisfy the $L^1$ MZ condition of degree $n$ with constant independent of $n$. Since $\phi_m\in\mathbb{P}_n(\mathbb{S}^2)$, $\psi_n=\phi_m^2$ and $\|\phi_m\|_{L^2(\mathbb{S^2)}} = 1$, we also have
\begin{equation} \label{eq:divergence}
\frac{\sum_{j=1}^{N_n} |w_j^{(n)}| |\phi_m(\mathbf{x}_j^{(n)})|^2}{\|\phi_m\|_{L^2(\mathbb{S}^2)}^2}
=
\frac{\sum_{j=1}^{N_n} |w_j^{(n)}| |\psi_n(\mathbf{x}_j^{(n)})|}{\|\psi_n\|_{L^1}} \gtrsim 1 + \sqrt{n}.
\end{equation}
Thus, the $L^2$ MZ constants cannot be uniformly bounded in $n$, either.
\end{proof}

\section{Numerical Validations}
\label{sec:validation}

In this section, we present some computational methodologies for evaluating the MZ constants discussed in previous sections, and empirically validate the strict logical hierarchy of the stability conditions established in \Cref{sec:hierarchy}.

\subsection{Computation and Approximation of the Optimal MZ Constants}
\label{sec:computation-of-MZ-constants}

For each $\mathbf{x} \in \mathbb{S}^2$, let the vector $\mathbf{y}_n(\mathbf{x}) \in \mathbb{R}^{(n+1)^2}$ denote the evaluation of all spherical harmonics up to degree $n$ at $\mathbf{x}$. That is,
\begin{equation*}
    \mathbf{y}_n(\mathbf{x}) := (Y_{0,0}(\mathbf{x}), Y_{1,-1}(\mathbf{x}), Y_{1,0}(\mathbf{x}), Y_{1,1}(\mathbf{x}), \ldots, Y_{n,n}(\mathbf{x}))^\top.
\end{equation*}
Every polynomial $p \in \mathbb{P}_n(\mathbb{S}^2)$ has a unique representation $p(\mathbf{x}) = \mathbf{a}^\top \mathbf{y}_n(\mathbf{x})$ for some coefficient vector $\mathbf{a} \in \mathbb{R}^{(n+1)^2}$. Orthonormality of spherical harmonics implies the Parseval's identity $$
\|p\|_{L^2(\mathbb{S}^2)} = \|\mathbf{a}\|_2.
$$ 

This coefficient representation reduces the computation of the three MZ constants to finite-dimensional optimization problems. The $L^2$ MZ constant is the largest eigenvalue of the Gram matrix, whereas the $L^1$--$L^2$ and $L^1$ MZ constants involve absolute values and are approximated numerically by smooth optimization problems on $\mathbb{S}^{(n+1)^2-1}$. For the latter purpose, we use the smooth approximation $\phi_\epsilon(t):=\sqrt{t^2+\epsilon^2}$ with $\epsilon > 0$ a smoothing parameter.

\subsubsection{Computation of the \texorpdfstring{$L^2$}{L2} MZ Constant}
The optimal $L^2$ MZ constant $c_2$ is the smallest constant such that
\begin{equation*}
\sum_{j=1}^N |w_j||p(\mathbf{x}_j)|^2 \leq c_2\|p\|_{L^2(\mathbb{S}^2)}^2 \quad \forall\, p \in \mathbb{P}_n(\mathbb{S}^2).
\end{equation*}
In terms of the coefficient representation $p = \mathbf{a}^\top \mathbf{y}_n$, by the homogeneity of both sides and Parseval's identity, we have
\begin{equation*}
    c_2 = \max_{\mathbf{a} \in \mathbb{S}^{(n+1)^2-1}} \sum_{j=1}^N |w_j||\mathbf{a}^\top \mathbf{y}_n(\mathbf{x}_j)|^2.
\end{equation*}
Define the weighted Gram matrix 
\begin{equation*}
    \mathbf{G}_n := \mathbf{Y}_n \mathrm{diag}(|w_1|, \ldots, |w_N|)\mathbf{Y}_n^\top,
\end{equation*}
where $\mathbf{Y}_n$ is as in \eqref{eq:basis-matrix}. Then, we have
\begin{equation*}
    \sum_{j=1}^N |w_j||\mathbf{a}^\top \mathbf{y}_n(\mathbf{x}_j)|^2 = \mathbf{a}^\top \mathbf{G}_n \mathbf{a}.
\end{equation*}
It follows that $c_2 = \lambda_{\max}(\mathbf{G}_n)$, the largest eigenvalue of $\mathbf{G}_n$.

\subsubsection{Approximation of the \texorpdfstring{$L^1$--$L^2$}{L1-L2} MZ Constant}
The optimal $L^1$--$L^2$ MZ constant $C$ is the smallest constant satisfying
\begin{equation*}
    \sum_{j=1}^N |w_j||p(\mathbf{x}_j)| \leq C\|p\|_{L^2(\mathbb{S}^2)} \quad \forall\, p \in \mathbb{P}_n(\mathbb{S}^2).
\end{equation*}
By the homogeneity of both sides and Parseval's identity, we have
\begin{equation*}
    C = \max_{\mathbf{a} \in \mathbb{S}^{(n+1)^2-1}} \sum_{j=1}^N |w_j||\mathbf{a}^\top \mathbf{y}_n(\mathbf{x}_j)|.
\end{equation*}
The objective is nonsmooth, we therefore approximate the $L^1$--$L^2$ MZ constant by numerically solving the following smooth surrogate:
\begin{equation} \label{eq:L1-L2-MZ-constant}
    C \approx \max_{\mathbf{a} \in \mathbb{S}^{(n+1)^2-1}} \sum_{j=1}^N |w_j|\phi_\epsilon(\mathbf{a}^\top\mathbf{y}_n(\mathbf{x}_j)).
\end{equation}

\subsubsection{Approximation of the \texorpdfstring{$L^1$}{L1} MZ Constant}
The optimal $L^1$ MZ constant $c_1$ is the smallest constant satisfying
\begin{equation*}
    \sum_{j=1}^N |w_j||p(\mathbf{x}_j)| \leq c_1 \|p\|_{L^1(\mathbb{S}^2)} \quad \forall\, p \in \mathbb{P}_n(\mathbb{S}^2).
\end{equation*}
Equivalently,
\begin{equation*}
    c_1 = \sup_{\substack{p \in \mathbb{P}_n(\mathbb{S}^2) \\ p \neq 0}} \frac{\sum_{j=1}^N |w_j||p(\mathbf{x}_j)|}{\|p\|_{L^1(\mathbb{S}^2)}}.
\end{equation*}
The quotient is homogeneous of degree zero in $p$. Hence, using the coefficient representation, the optimization may be restricted to the hypersphere, and we obtain
\begin{equation*}
    c_1 = \max_{\mathbf{a} \in \mathbb{S}^{(n+1)^2-1}} \frac{\sum_{j=1}^N |w_j||\mathbf{a}^\top\mathbf{y}_n(\mathbf{x}_j)|}
    {\|\mathbf{a}^\top \mathbf{y}_n\|_{L^1(\mathbb{S}^2)}}.
\end{equation*}
In general, the denominator of the quotient does not admit an analytic expression. We numerically approximate it using a positive Lebedev quadrature formula exact of degree $131$ with nodes $\{\mathbf{z}_k\}_{k=1}^K$ and weights $\{\lambda_k\}_{k=1}^K$ taken from \cite{sphere_lebedev_rule}. We approximate $\|\mathbf{a}^\top \mathbf{y}_n\|_{L^1(\mathbb{S}^2)} \approx \sum_{k=1}^K \lambda_k |\mathbf{a}^\top \mathbf{y}_n(\mathbf{z}_k)|$.

To obtain a differentiable smooth objective, we further smooth the absolute values in both the numerator and denominator, and approximate the $L^1$ MZ constant by solving
\begin{equation} \label{eq:L1-MZ-constant}
    c_1 \approx \max_{\mathbf{a} \in \mathbb{S}^{(n+1)^2-1}} \frac{\sum_{j=1}^N |w_j| \phi_\epsilon(\mathbf{a}^\top \mathbf{y}_n(\mathbf{x}_j))}{\sum_{k=1}^K \lambda_k \phi_\epsilon(\mathbf{a}^\top \mathbf{y}_n(\mathbf{z}_k))}.
\end{equation}
Since the Lebedev quadrature formula is exact for constant functions and $\phi_\epsilon(t) \geq \epsilon$ for all $t \in \mathbb{R}$, the denominator of the above quotient is always bounded below by $\epsilon$.

\subsection{Numerical Verifications of the Strict Logical Hierarchy}

We apply the computational methodologies introduced in \Cref{sec:computation-of-MZ-constants} to numerically validate the strict logical hierarchy of the stability conditions shown in \cref{fig:hierarchy}. To approximate the $L^1$ MZ constant $c_1$ and the $L^1$--$L^2$ MZ constant $C$, we choose $\epsilon = 10^{-6}$ as the smoothing parameter and employ the manifold optimization toolbox Manopt \cite{manopt} to solve the optimization problems on hyperspheres. Since \eqref{eq:L1-L2-MZ-constant} and \eqref{eq:L1-MZ-constant} are nonconvex, the solver may return only a local maximizer or another stationary point. We consider the following three sequences of quadrature formulas:
\begin{enumerate}[(1), noitemsep]
    \item \textit{Spherical $n$-design quadrature formulas} $\{Q_n^1\}$: A sequence of baseline quadrature formulas where each $X^{(n)}$ is a spherical $n$-design of $N_n$ points (available through the public repository \cite{womersley2015efficient}) and all quadrature weights are equal: $w_j^{(n)} = 4\pi/N_n$.
    \item \textit{Dirac-like quadrature formulas} $\{Q_n^2\}$: A sequence of quadrature formulas defined in \cref{prop:Polya-does-not-imply-L1-L2-MZ}, where each $X^{(n)}$ is also a spherical $n$-design, but with all weights concentrated on a single node.
    \item \textit{Perturbed quadrature formulas} $\{Q_n^3\}$: A sequence of quadrature formulas defined in \cref{thm:hierarchy}, where each $X^{(n)}$ is a spherical $(2n+2)$-design of $N_n$ points containing the north pole, but whose weights are perturbed by a spherical harmonic, $w_j^{(n)} = (4\pi/N_n)(1+Y_{n+1,0}(\mathbf{x}_j^{(n)}))$.
\end{enumerate}

\begin{figure}[htbp]
\centering
\begin{subfigure}[b]{0.3\textwidth}
    \centering
    \includegraphics[width=\textwidth]{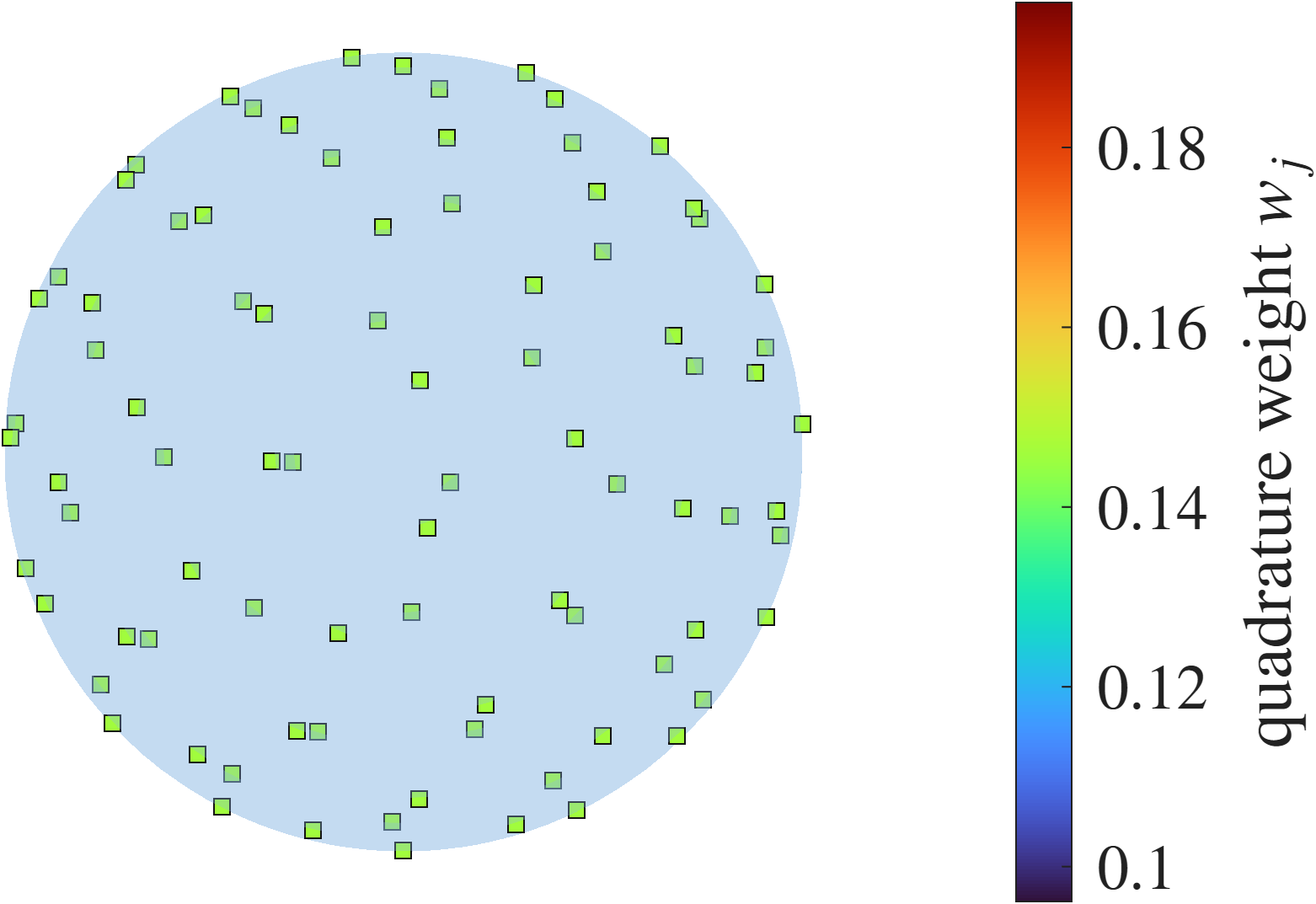} 
    \caption{$Q_{12}^1$}
\end{subfigure}
\hfill
\begin{subfigure}[b]{0.3\textwidth}
    \centering
    \includegraphics[width=\textwidth]{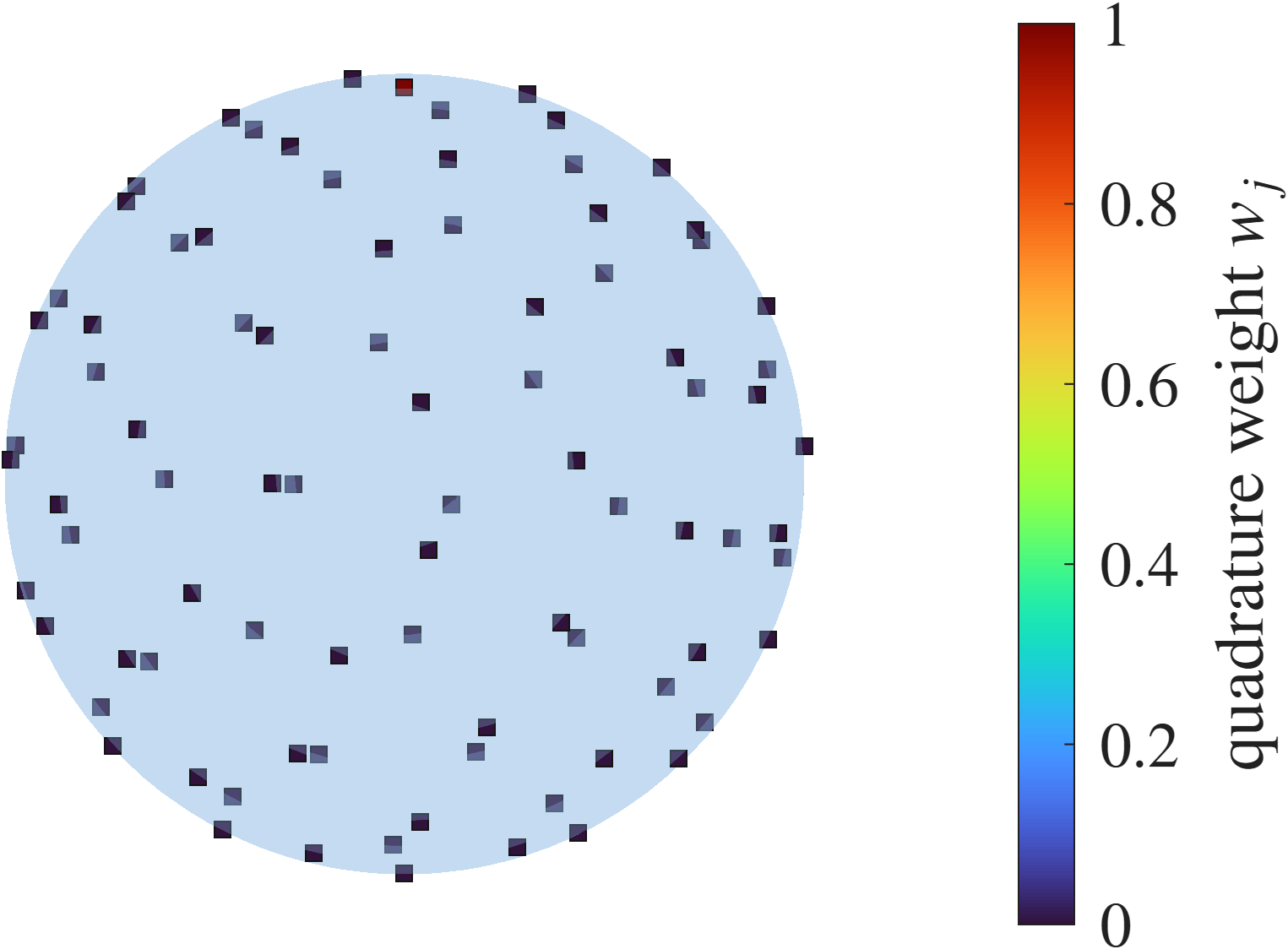} 
    \caption{$Q_{12}^2$}
\end{subfigure}
\hfill
\begin{subfigure}[b]{0.3\textwidth}
    \centering
    \includegraphics[width=\textwidth]{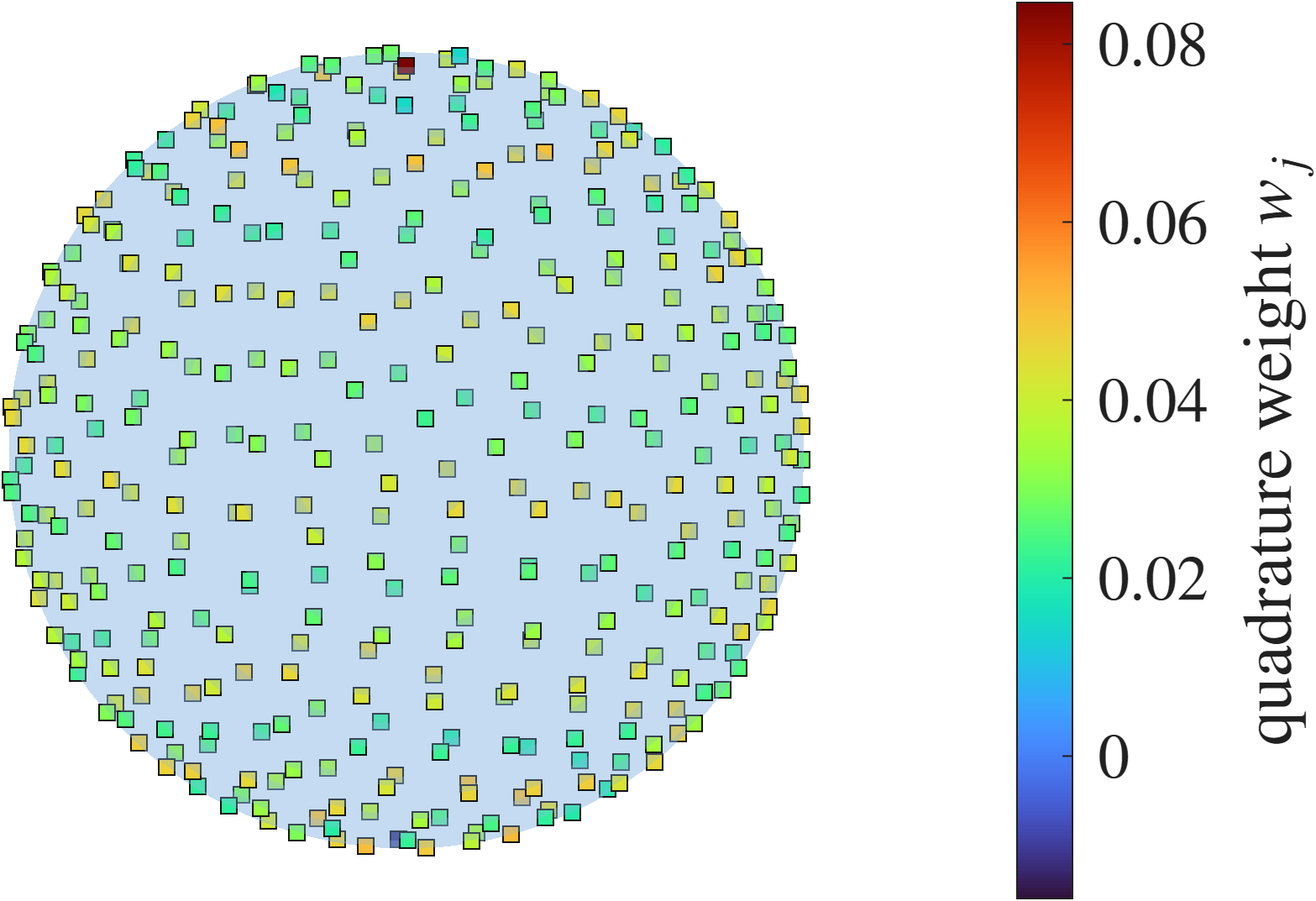}
    \caption{$Q_{12}^3$}
\end{subfigure}
\caption{Visualization of the node distributions and quadrature weights on the unit sphere for the three representative quadrature formulas at index $n = 12$.}
\label{fig:node_distributions}
\end{figure}

We visualize the node distributions and quadrature weights of a representative configuration from each sequence at index $n = 12$ in \cref{fig:node_distributions}. The baseline spherical $n$-design quadrature formula $Q_{12}^1$ exhibits a uniform weight distribution over the sphere. The Dirac-like quadrature formula $Q_{12}^2$ has an identical node distribution to $Q_{12}^1$, but concentrates its entire weights on a single node, leaving all other weights at zero. The perturbed quadrature formula $Q_{12}^3$ displays an oscillatory pattern introduced by $Y_{13,0}(\mathbf{x})$ along latitude lines.

The stability constants $M$, $C$, $c_1$, and $c_2$ are plotted with respect to the sequence index $n$ in \cref{fig:constant_growth}. These numerical results are consistent with the strict logical hierarchy in Figure \ref{fig:hierarchy}. For the spherical $n$-design quadratures $Q_n^1$, all four constants remain approximately constant over the tested range, in agreement with the stability of these rules. For the Dirac-like quadratures $Q_n^2$, the P\'olya constant $M$ remains equal to $1$, whereas the MZ constants $C$, $c_1$, and $c_2$ grow rapidly, confirming that the uniform P\'olya stability condition does not control stability of higher degree polynomials. For the perturbed quadratures $Q_n^3$, the $L^1$--$L^2$ MZ constant $C$ remains uniformly bounded above by $\sqrt{4\pi}+1\approx4.54$, as proved in \cref{thm:hierarchy}, while the estimates of $c_1$ and $c_2$ exhibit growth consistent with the divergent lower bounds in \cref{eq:asymptotic-lower-bound,eq:divergence}.

\begin{figure}[htbp]
\centering
\begin{subfigure}[b]{0.8\textwidth}
    \centering
    \includegraphics[width=\textwidth]{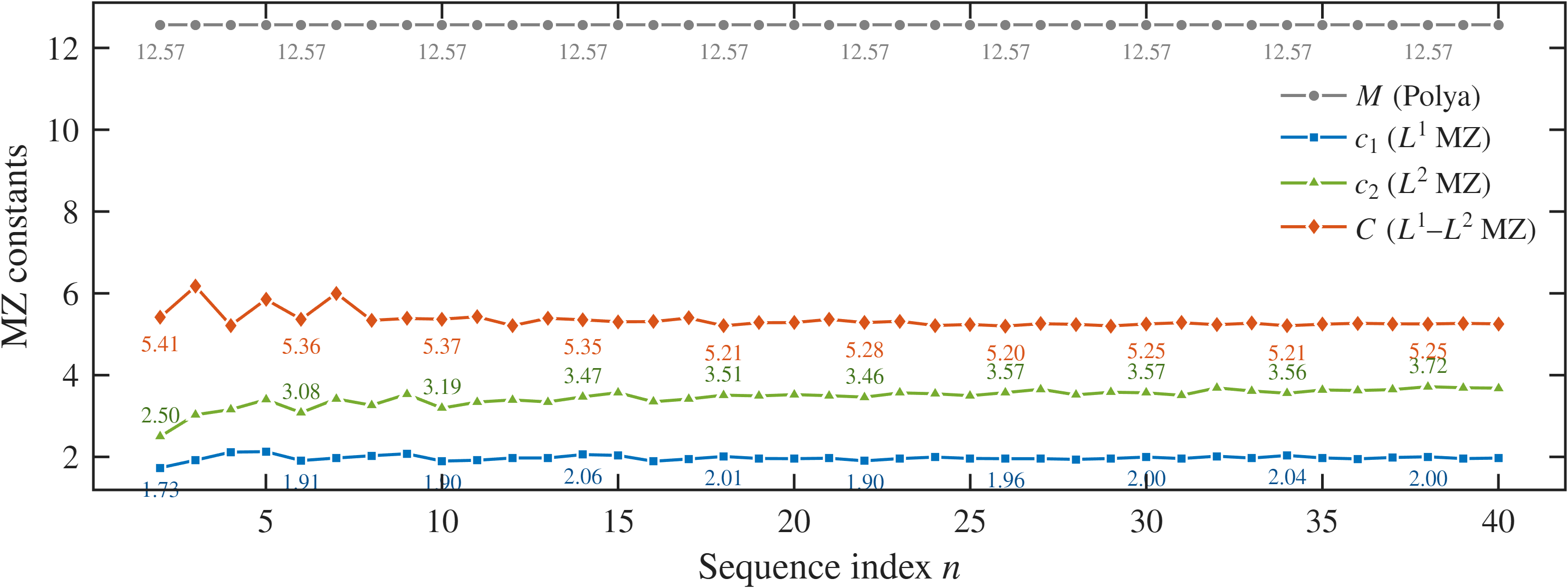}
    \caption{Spherical $n$-design quadrature formulas}
\end{subfigure}

\begin{subfigure}[b]{0.8\textwidth}
    \centering
    \includegraphics[width=\textwidth]{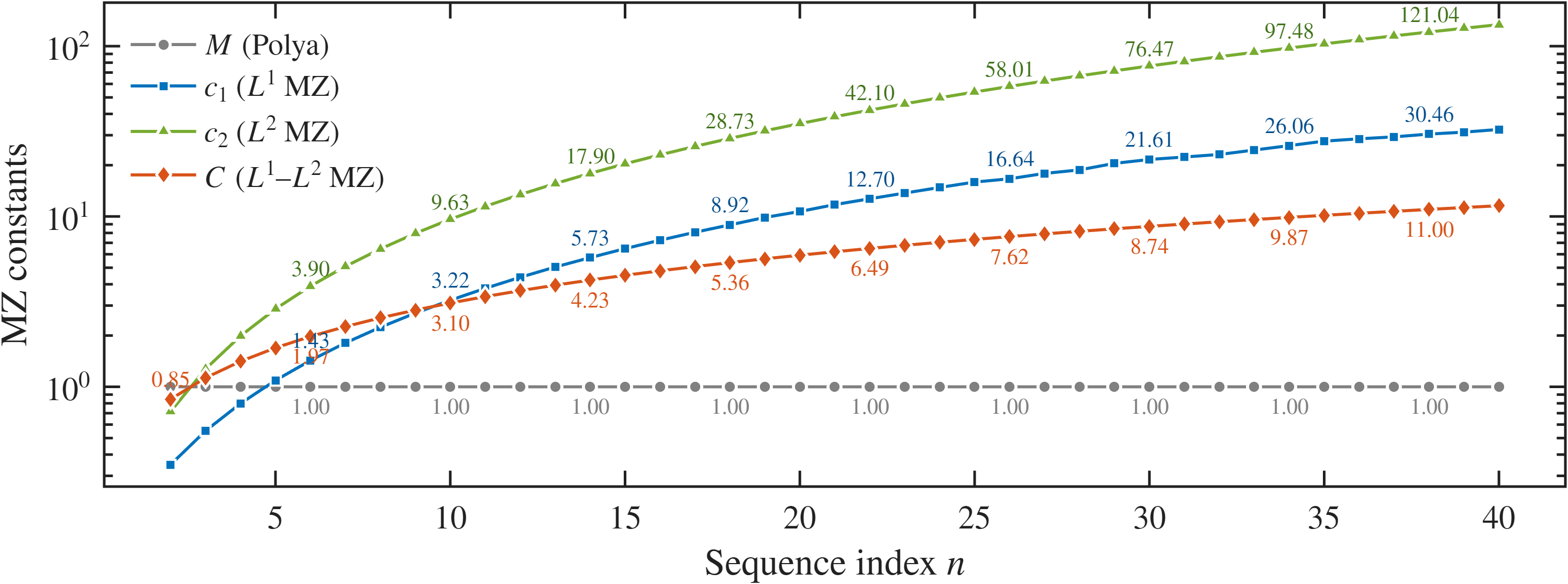}
    \caption{Dirac-like quadrature formulas}
\end{subfigure}

\begin{subfigure}[b]{0.8\textwidth}
    \centering
    \includegraphics[width=\textwidth]{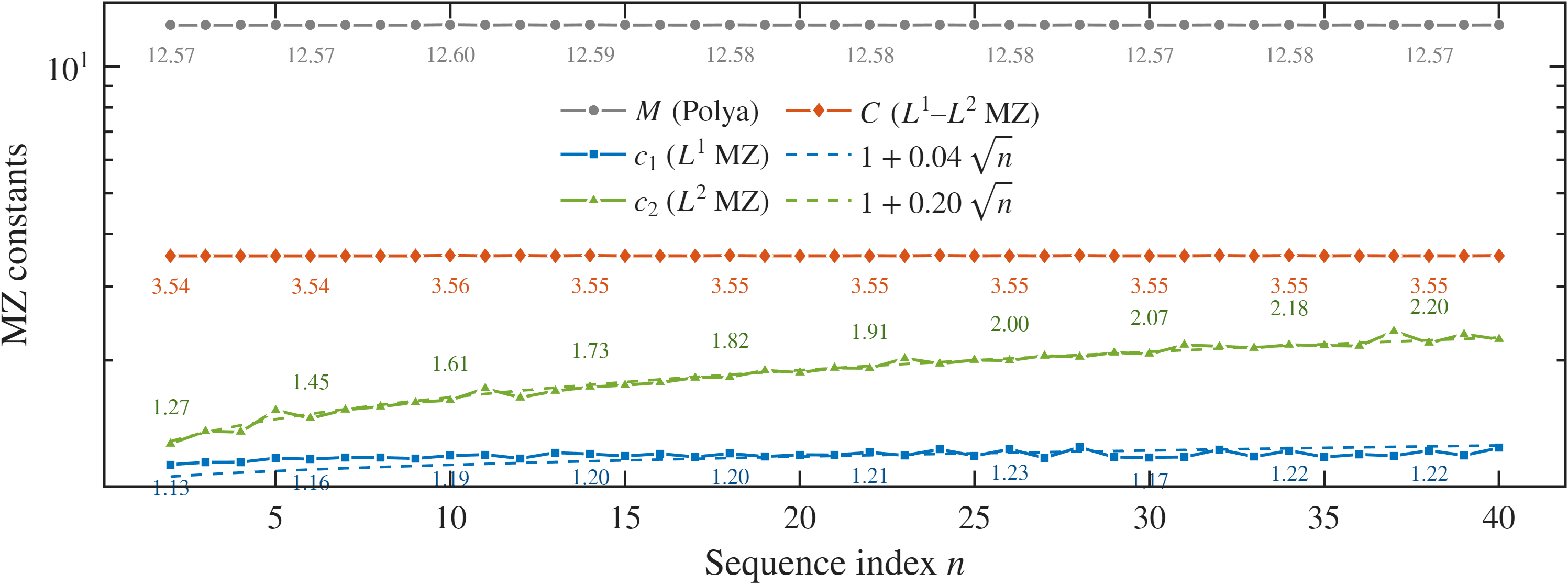}
    \caption{Perturbed quadrature formulas}
\end{subfigure}
\caption{Growth of the stability constants $M$, $C$, $c_1$, and $c_2$ plotted against the sequence index $n$ for the three sequences of quadrature formulas.}
\label{fig:constant_growth}
\end{figure}

\section{Conclusions}
\label{sec:conclusion}
In this paper, we answered the question about the necessary and sufficient conditions for the $L^2$ convergence of hyperinterpolation, and obtained a complete characterization on the unit sphere. We showed that the convergence for every continuous function is equivalent to the uniform $L^1$--$L^2$ MZ condition together with the asymptotic functional approximation property for spherical polynomials. We also proved that the stability condition is sharp in the sense that the optimal $L^1$--$L^2$ MZ constant is exactly the operator norm of the corresponding hyperinterpolation operator, and this identity is naturally explained through Banach space duality. We also constructed explicit counterexamples showing that P\'olya's classical conditions for quadrature convergence do not suffice for hyperinterpolation convergence. These results reveal a genuine distinction between the convergence of quadrature functionals and that of the associated approximation operators, and lead to a strict logical hierarchy among the stability and accuracy conditions pertinent to quadrature and hyperinterpolation convergence. We further numerically illustrated the separations established by the theoretical analysis.

Although the present work is formulated on $\mathbb{S}^2$, the functional-analytic mechanism underlying the stability characterization is not inherently spherical. It relies principally on finite-dimensional reproducing approximation spaces, the Hilbert-space geometry of the target norm, and the duality between continuous functions and finite signed measures. This suggests that an analogous convergence theory can be developed on compact Riemannian manifolds by replacing spherical polynomial spaces with suitable spectral subspaces, provided that the corresponding reproducing kernels, MZ sampling inequalities, and asymptotic quadrature approximation properties are available. Establishing such a manifold-level theory, together with constructive geometric conditions for stable sampling and quadrature, is a natural direction for future research.

\bibliographystyle{siamplainmc}
\small
\bibliography{references}

\end{document}